\documentclass[12pt,reqno]{amsart}
\usepackage[margin=1in]{geometry}
\makeatletter
\def\section{\@startsection{section}{1}%
	\z@{.7\linespacing\@plus\linespacing}{.5\linespacing}%
	{\bfseries%\normalfont\scshape
		\centering
}}
\def\@secnumfont{\bfseries}
\makeatother
\usepackage{amsmath,amssymb,amsthm,graphicx,amsxtra, setspace}
\usepackage[utf8]{inputenc}
\usepackage{charter}
\usepackage{mathrsfs}
\usepackage{alltt}
\usepackage{stackengine}
\usepackage{relsize}
\usepackage{hyperref}
\usepackage{aliascnt}
\usepackage{tikz}
\usepackage{mathtools}
\usepackage{multicol}
\usepackage{upgreek}
\usepackage{graphicx,type1cm,eso-pic,color}
\allowdisplaybreaks
\usepackage{scalerel,stackengine}
\stackMath
\usepackage{setspace}
\newtheorem{theorem}{Theorem}[section]
\newtheorem*{theorem*}{Theorem}

\newaliascnt{lemma}{theorem}
\newtheorem{lemma}[lemma]{Lemma}
\aliascntresetthe{lemma} 

\newaliascnt{proposition}{theorem}
\newtheorem{proposition}[proposition]{Proposition}
\aliascntresetthe{proposition}

\newaliascnt{assumption}{theorem}

\aliascntresetthe{assumption}

\newaliascnt{auxiliary}{theorem}

\aliascntresetthe{auxiliary}

\newaliascnt{corollary}{theorem}

\aliascntresetthe{corollary}

\newaliascnt{definition}{theorem}
\newtheorem{definition}[definition]{Definition}
\aliascntresetthe{definition}

\newaliascnt{example}{theorem}

\aliascntresetthe{example}

\newaliascnt{remark}{theorem}
\newtheorem{remark}[remark]{Remark}
\aliascntresetthe{remark}

\newaliascnt{hypothesis}{theorem}

\aliascntresetthe{hypothesis}

\newaliascnt{property}{theorem}

\aliascntresetthe{property}

\usepackage[utf8]{inputenc}
\usepackage{amsmath}
\usepackage{enumerate}
\usepackage{bm}
\usepackage[T1]{fontenc}
\usepackage{amsthm}
\usepackage{amsfonts}
\usepackage{amssymb}
\usepackage{graphicx}
\usepackage{enumerate}
\usepackage[english]{babel}
\usepackage{comment}
\usepackage{amssymb}
\usepackage{amsthm}
\usepackage{xcolor}
\usepackage{dutchcal}

\usepackage[title]{appendix}
\newcommand{\RR}{\mathbb{R}}
\newcommand{\Om} {\Omega}
\newcommand{\pa} {\partial}
\newcommand{\be} {\begin{equation}}
\newcommand{\ee} {\end{equation}}
\newcommand{\bea} {\begin{eqnarray}}
\newcommand{\eea} {\end{eqnarray}}
\newcommand{\Bea} {\begin{eqnarray*}}
\newcommand{\Eea} {\end{eqnarray*}}

\usepackage{xcolor}
\newcommand{\al} {\alpha}
\newcommand{\ba} {\beta}

\newcommand{\ga} {\gamma}

\newcommand{\la} {\lambda}

\newcommand{\La} {\Lambda}

\newcommand{\noi} {\noindent}

\newcommand{\va} {\varphi}

\newcommand{\e}{\epsilon}

\newcommand{\fpl}{(-\Delta)_p^s\,}
\newcommand{\pst}{p_s^*}
\newcommand{\DD} {\displaystyle}
\newcommand {\n}{\nonumber\\}
\newcommand{\Addresses}{{% additional braces for segregating \footnotesize
		\footnote{
				\footnotesize
\noindent \textsuperscript{1}Indian Institute of Science Education and Research, Thiruvananthapuram 695551, India  \par\nopagebreak 
   \noindent \textsuperscript{2}Department of Mathematics, Indian Institute of Technology, Guwahati-781039, India  \par\nopagebreak 
\noindent \textsuperscript{A}\textit{e-mail:} \texttt{dhanya.tr@iisertvm.ac.in}.
\noindent \textsuperscript{B}\textit{e-mail:} \texttt{ritabrata20@iisertvm.ac.in}.
\noindent \textsuperscript{C}\textit{e-mail:} \texttt{uttam.maths@iitg.ac.in}.
\noindent \textsuperscript{D}\textit{e-mail:} \texttt{swetatiwari@iitg.ac.in}.
			\medskip\noindent
			{\bf Acknowledgments:} 
 		R. Dhanya was supported by INSPIRE faculty fellowship with grant number DST/INSPIRE/04/2015/003221 when the work was being carried out.  Ritabrata Jana is currently supported by the Prime Minister Research Fellowship during the execution of this research.
			
}}}
\begin{document}
\title[Positive solutions for fractional p- Laplace semipositone problem]{ Positive solutions for fractional p- Laplace semipositone problem with superlinear growth 	\Addresses	}
	\author[ R. Dhanya ]
	{ R. Dhanya\textsuperscript{1,A}} 
 	\author[Ritabrata Jana]
	{Ritabrata Jana\textsuperscript{1,B}} 
 \author[Uttam Kumar]
	{Uttam Kumar\textsuperscript{2,C}} 
 \author[Sweta Tiwari]
	{Sweta Tiwari\textsuperscript{2,D}} 
\maketitle
\begin{abstract}
We consider a semipositone problem involving the fractional $p$ Laplace operator of the form
\begin{equation*}
\begin{aligned}
(-\Delta)_p^s u &=\mu( u^{r}-1) \text{ in } \Omega,\\
u &>0  \hspace{1.40cm} \text{ in }\Omega,\\
u &=0 \hspace{1.40cm} \text{ on }\Omega^{c},
\end{aligned}
\end{equation*}
where $\Omega$ is a smooth bounded convex domain in $\RR^N$, $p-1<r<p^{*}_{s}-1$, where $p_s^{*}:=\frac{Np}{N-ps}$, and $\mu$ is a positive parameter. We study the behaviour of the barrier function under the fractional $p$-Laplacian and use this information to prove the existence of a positive solution for small $\mu$ using degree theory. Additionally, the paper explores the existence of a ground state positive solution for a multiparameter semipositone problem with critical growth using variational arguments.
\end{abstract}
\keywords{\textit{Key words:} Nonlocal problem, fractional p- Laplacian, semipositone problem, indefinite nonlinearities, critical exponent}
% \\
% MSC(2020): 35B06, 35B51, 35J92
\section{Introduction}
In this paper, we consider a fractional  $p$-Laplace equation with a superlinear nonlinearity that changes its sign. We consider $\Omega$ to be a bounded  convex open  set in $\RR^N,$  $N\geq2$ with $C^{2}$-boundary. Our primary focus is to obtain a positive solution to the following Dirichlet boundary value problem
 \begin{eqnarray} \label{pmu}(P^\mu)  \left\{
 	\begin{array}{rl}
 		(-\Delta)_p^s u &=\mu( u^{r}-1) \text{ in }\Omega, \\
 		u &>0  \hspace{1.40cm} \text{ in }\Omega,\\
 		u &=0 \hspace{1.40cm} \text{ on }\Omega^{c},
 	\end{array} 
 	\right. 
 \end{eqnarray}
where $p-1<r< p^{*}_{s}-1,  N > p s$ with
 $p^{*}_{s}:=\frac{Np}{N-ps}$ is the fractional critical Sobolev exponent. Upto a suitable normalization constant,  the operator fractional $p$-Laplacian $(-\Delta)_{p}^s u$, 
is defined as
\begin{equation}\label{o11}
(-\Delta)_{p}^su(x)= 2\thinspace PV. \int_{\mathbb{R}^N}\frac{|u(x)-u(y)|^{p-2}(u(x)-u(y)}{|x-y|^{N+sp}} \,dy \qquad x\in \Omega.
\end{equation}
\par  Motivation to study the semipositone problems stems from pure mathematical interest and also due to its applications in various other areas such as mathematical physics, population dynamics, and materials science (see \cite{CMS00} and references therein). Semipositone problems are characterized by a nonlinearity that changes sign, making the existence of a positive solution difficult to establish. In \cite{KS92}, authors observed a striking phenomenon that non-negative solutions of semipositone problems can have interior zeros. As a result of this intriguing finding, researchers have focused their attention in proving the existence of positive solutions to semipositone problems. This topic has been extensively studied for Laplace and p Laplace operators in  \cite{DH96,HS03,PSS20,RSY07} and many more. The recent advances in the study of nonlocal operators, particularly in understanding fine boundary regularity \cite{IMSS1} and Hopf-type lemmas \cite{LZ21}, have played a crucial role in exploring nonlocal semipositone problems. As a result, researchers have been able to use these tools to establish the positivity of solutions in recent works, including \cite{drst21} and \cite{lopera22}. The paper \cite{drst21} employs a sub-supersolution approach to demonstrate the existence of a solution, while \cite{lopera22} relies on the mountain pass theorem.

%Another commonly used technique for proving the existence of a positive solution is to construct appropriate ordered pair of positive sub-super solutions. However, this method is complicated since we need  pointwise estimates of the sub-solution (or super solution), which becomes especially challenging when the non-local operator is defined via global integration. 

In this article, we utilize degree theoretic arguments to demonstrate the existence of a positive solution to the superlinear semipositone problem $(P^\mu)$. To accomplish this, we first establish apriori $L^\infty$ estimates for a family of subcritical problems.  This inturn
reduces to determining the pointwise bounds for the behavior of a barrier function near the boundary of $\Omega$ when acted upon by the fractional p Laplacian. The novelty of this work is to overcome the difficulty which arises due to the global integration as well as the nonlinearity of the operator while obtaining this pointwise bound. We anticipate that this estimate may have additional applications, such as providing an effective upper bound for solutions by utilizing relevant comparison principles and thereby establishing regularity results.  
We next state our main theorem which is proved in Section 5 of this paper.
 \begin{theorem}\label{t61}
 	For $p\geq 2$, there exists $\mu_{0}>0$ such that the problem $(P^\mu)$ admits a positive solution for $\mu\in (0,\mu_{0})$ and $p-1<r<\pst-1$.
 \end{theorem}
The use of apriori estimates for nonlinear equations is known to be a vital tool in proving the existence of solutions to nonlinear problems. This estimate is often used to verify assumptions of the Leray Schauder degree theory about the existence of a fixed point. Chen et al.\cite{chen2016direct} demonstrated the effectiveness of this approach in proving the existence of solutions to fractional Laplace equation. Barrios et al. \cite{BLG} examined more general nonlocal operators and considered the viscosity solutions of the problem in the presence of a gradient term. Both proofs rely on constructing a barrier function and analyzing its behavior near the boundary. In this article, we use the idea of  Gidas-Spruck translated function to obtain the uniform $L^\infty$ estimate for viscosity solutions by studying a proper barrier function under the fractional p-Laplacian. 
It is well known that Gidas-Spruck type estimate transforms the uniform $L^\infty$ bound to that of Liouville type result. In order to obtain the results outlined in Theorem \ref{t61}, we must rely on a nonexistence assumption, denoted $(\mathcal{NA})$, which is similar to the nonexistence assumption made by Brasco et al. in \cite{brasco}. We consider a sub-critical exponent problem with $p-1< q< \pst-1$ either when $\mathcal{H}$ is the entire space $\mathbb R^N$ or a half-space  in $\mathbb R^N$ given by 
\begin{eqnarray}\label{p62}  \left\{
	\begin{array}{rl}
		(-\Delta)_p^s u &= u^{q} \text{ in }\mathcal{H}, \\
	u&>0\hspace{0.2cm} \mbox{ in } \mathcal{H};\\	u &=0 \hspace{0.20cm} \text{ on }\mathbb R^N\setminus\mathcal{H}. 
	\end{array} 
	\right. 
\end{eqnarray}
\begin{itemize}
    \item [$(\mathcal{NA})$] Problem (\ref{p62}) does not admit nontrivial viscosity solution in $C^{\alpha}(\mathcal{H}).$
\end{itemize}
Although to the best of our knowledge this result has not been proven yet, we have a strong basis for believing that it is a plausible assumption. Specifically, we note that analogous results have been established for both the fractional Laplacian (i.e., for $0<s<1$ and $p=2$) in Theorem 1.1 and 1.2 of Quass and Xia\cite{qx}, and for the local case ( $-\Delta_p , \, 1<p<\infty$) in Theorem 1.1 and Lemma 2.8 of Zou \cite{zhh}. Therefore, since nonexistence results are already known for these special cases, it is reasonable to assume $\mathcal{(NA)}.$ 
% \noi For the case of p-Laplacian, Zou\cite[Theorem 1.1]{zhh} has proved the nonexistence of a more general case for the halfspace. We can conclude a similar result for the whole space using \cite[Lemma 2.8]{zhh}. For the fractional case, \cite[Theorem 1.1 and Theorem 1.2]{qx} gives the nonexistence. For the fractional p-Laplace equation it is not proved yet up to our knowledge.
% In a recent preprint, Lopera et al., addressed the same problem $(P^\mu)$ where they obtained the existence of a solution combining the variational method with the regularity results. Though the authors have proved the existence of a solution for all $r\in (p-1,p_s^*-1)$ and $\mu$ small enough,  the important question which is the  positivity of the solution was obtained only with some restricted hypothesis on the exponent $r.$ We wish to emphasize that we don't have any restriction on the exponent "r."
\par
In a recent preprint, Lopera et al. \cite{lopera22} investigated a semipositone problem with superlinear nonlinearity that is similar to $(P^\mu)$. They utilized the variational method to establish the existence of a nonnegative solution when $r$ belongs to the interval $(p-1,p_s^*-1)$ and $\mu$ is sufficiently small. However, their findings only guarantee the positivity of the solution for specific values of $r$ and small $\mu$. In contrast, our solution to problem $(P^\mu)$ does not impose any such restrictions and guarantees the positivity of the solution for all values of $r$. 

In the last section of this article, we consider a multiparameter semipositone problem involving critical Sobolev exponent given as below: 
\begin{eqnarray}\label{nscp5}  \left.
	\begin{array}{rl}
		(-\Delta)_{p}^s u &=\lambda u^{p-1}+\mu(u^{p^{*}_{s}-1}-1) \text{ in }\Omega, \\
		u &>0  \hspace{3.55cm} \text{ in }\Omega,\\
		u &=0 \hspace{3.55cm} \text{ on }\Omega^{c},
	\end{array} 
	\right\} 
\end{eqnarray} 
\noindent
\noi When $\la=0$ and $\mu>0,$ the problem (\ref{nscp5}) reduces to (\ref{pmu})  for $r=p^{*}_s-1.$ We observe that the scaling $u\mapsto \mu^{\frac{1}{p^{*}_{s}-p}}u$ transforms the  equation in the following critical semipositone nonlocal problem which we call as $(P_\la^\mu)$

\begin{eqnarray}\label{nscp3} (P^\mu_\lambda)  \left\{
	\begin{array}{rl}
		(-\Delta)_{p}^s u &=\lambda u^{p-1}+u^{p^{*}_{s}-1}-\mu \text{ in }\Omega, \\
		u &>0  \hspace{3cm} \text{ in }\Omega,\\
		u &=0 \hspace{3cm} \text{ on }\Omega^{c}.
	\end{array} 
	\right. 
\end{eqnarray}
where $\lambda,\mu >0$ are parameters.
\noindent
For a given $\lambda>0$, the solution of the problem \eqref{nscp3}  are the critical points of the energy functional $E_{\mu}:D_{0}^{s,p}(\Omega)\rightarrow \mathbb{R}$
defined by  
\begin{equation*}
	E_{\mu}(u)=	 \frac{1}{p} \int_{\mathbb{R}^{N} \times \mathbb{R}^{N}} \dfrac{|u(x)-u(y)|^{p}}{|x-y|^{N+sp}}\,dx\, dy - \frac{\lambda}{p}\int_{\Omega}u^{p}\,dx- \frac{1}{p^{*}_{s}} \int_{\Omega} u^{p^{*}_{s}}\,dx +\int_{\Omega}\mu u\,dx.
\end{equation*}
All the weak solutions of problem \eqref{nscp3} lie on the set
\begin{equation*} 
\mathcal{N}_{\mu}=
\left\{
\begin{aligned}
u\in D_{0}^{s,p}(\Omega): &\ u>0 \ \text{in}\  \Omega \thickspace
	\thickspace  \text{and} \\ &\int_{\mathbb{R}^N \times \mathbb{R}^N} \dfrac{|u(x)-u(y)|^{p}}{|x-y|^{N+sp}}\,dx\, dy = \int_{\Omega}(\lambda u^{p}+u^{p^{*}_{s}}-\mu u)\,dx 
\end{aligned}   
\right\}
\end{equation*}

% \begin{equation*}
% 	\mathcal{N}_{\mu}=\left\lbrace u\in D_{0}^{s,p}(\Omega):u>0 \thickspace \text{in}\thickspace  \Omega \thickspace \text{and} 
% 	\thickspace  \int_{\mathbb{R}^N \times \mathbb{R}^N} \dfrac{|u(x)-u(y)|^{p}}{|x-y|^{N+sp}}\,dx\, dy = \int_{\Omega}(\lambda u^{p}+u^{p^{*}_{s}}-\mu u)\,dx\right\rbrace .	
% \end{equation*}
A weak solution that minimizes $E_{\mu}$ on $\mathcal{N}_{\mu}$ is a ground state solution for \eqref{nscp3}. 
Let 
\begin{equation}\label{nscp4}
	\lambda_{1}=\inf_{u\in D^{s,p}_{0}(\Omega)\setminus{\{0\}}}	\dfrac{\displaystyle\int_{\mathbb{R}^N \times \mathbb{R}^N} \dfrac{|u(x)-u(y)|^{p}}{|x-y|^{N+sp}}\,dx\, dy}{\int_{\Omega}|u(x)|^{p}\,dx}
\end{equation}
is the first Dirichlet eigenvalue of the fractional $p$ Laplacian, which is positive. We prove that,
\begin{theorem}\label{nscpt1}
For $p\ge2$,  $N\ge sp^{2}$ and $\lambda\in (0,\lambda_{1})$, there exists $\mu^{*}>0$ such that  $\mu\in (0,\mu^{*})$, problem \eqref{nscp3} has a ground state solution $u_{\mu} \in C_{d}^{0,\alpha}(\overline{\Omega})$ for some $\alpha \in(0,s]$.	
\end{theorem}

% \begin{eqnarray}\label{nscp5}  \left.
% 	\begin{array}{rl}
% 		(-\Delta)_{p}^s u &=\lambda u^{p-1}+\mu(u^{p^{*}_{s}-1}-1) \text{ in }\Omega, \\
% 		u &>0  \hspace{3.55cm} \text{ in }\Omega,\\
% 		u &=0 \hspace{3.55cm} \text{ on }\Omega^{c},
% 	\end{array} 
% 	\right\} 
% \end{eqnarray} 
% \noindent
% into
% \begin{equation*}
% 	(-\Delta)_{p}^s u = \lambda u^{p-1}+u^{p^{*}_{s}-1}-\mu^{\frac{p^{*}_{s}-1}{p^{*}_{s}-p}}.
% \end{equation*}
% So throughout the paper, we study the equation \eqref{nscp5} to get the theorem \ref{nscpt1}.
\section{Preliminaries}
\subsection{Function spaces}
To begin, we will revisit the definitions of several Sobolev spaces with fractional orders that are utilized in this article. For a smooth, bounded domain $\Omega \subset \mathbb{R}^N$ (where $N\geq 2$)  we denote the standard $L^{p}(\Omega)$ norm by $|\cdot|_{L^{p}(\Omega)}$, where $p\in [1,\infty]$. For a measurable function $u:\mathbb{R}^{N}\rightarrow \mathbb{R}$ and for $p\in(1,\infty)$ and $s\in (0,1)$, let
\begin{equation*}
	[u]_{W^{s,p}(\mathbb{R}^N)}:=  \left(\int_{\mathbb{R}^N \times \mathbb{R}^N} \dfrac{|u(x)-u(y)|^{p}}{|x-y|^{N+sp}}\,dx\, dy\right)^{1/p}
\end{equation*}
be the Gagliardo seminorm.
For $sp<N$, we consider the space 
\begin{equation*}
	W^{s,p}(\mathbb{R}^{N}):= \left \{u \in {L}^{p}(\mathbb{R}^{N}):[u]_{{W^{s,p}(\mathbb{R}^{N})}}<\infty  \right\}.
\end{equation*}
Then $W^{s,p}(\mathbb{R}^{N})$ is a Banach space with respect to the norm 
\begin{equation*}
\|u\|_{{W^{s,p}(\mathbb{R}^{N})}}=\left( \|u\|^{p}_{L^{p}(\mathbb{R}^{N})} + [u]^{p}_{{W^{s,p}(\mathbb{R}^{N})}}\right)^{\frac{1}{p}} . 	
\end{equation*}
 To address the Dirichlet boundary condition, we consider the space
\begin{equation*}
	W_{0}^{s,p}(\Omega):= \left \{u \in W^{s,p}(\mathbb{R}^{N}):u=0\medspace\text{in}\medspace \mathbb{R}^{N}\setminus \Omega \right\},
\end{equation*}
which is a Banach space endowed with the norm $\|\cdot\|=	\|\cdot\|_{{W^{s,p}(\mathbb{R}^{N})}}$. Moreover the embedding $W^{s,p}_{0}(\Omega)\hookrightarrow L^{r}(\Omega)$ is continuous for $1\leq r\leq p^{*}_{s}$ and
compact for  $1\leq r < p^{*}_{s}$. 
% We recall that the space
%  $W^{s,p}_{0}(\Omega)$ can also be defined as completion of $C^{\infty}_{0}(\Omega)$ in the norm 
%  %$\|\cdot\|_{{W^{s,p}(\mathbb{R}^{N})}}$. 
 Due to continuous embedding of $W^{s,p}_{0}(\Omega)\hookrightarrow L^{r}(\Omega)$ for $1\leq r\leq p^{*}_{s}$, we define the equivalent norm on $W^{s,p}_{0}(\Omega)$ as
\begin{equation*}
	\|u\|_{W^{s,p}_{0}}:=\left(\int_{\mathbb{R}^N \times \mathbb{R}^N} \dfrac{|u(x)-u(y)|^{p}}{|x-y|^{N+sp}}\,dx\, dy\right)^{1/p}.
\end{equation*}
In order to handle the critical growth in problem (\ref{nscp5}), we will work with the following function spaces\begin{equation*}
	D^{s,p}(\mathbb{R}^{N}):= \left \{u \in {L}^{p_{s}^{*}}(\mathbb{R}^{N}):[u]_{{W^{s,p}(\mathbb{R}^{N})}}<\infty  \right\},
\end{equation*}
\begin{equation*}
	D_{0}^{s,p}(\Omega):= \left \{u \in D^{s,p}(\mathbb{R}^{N}):u=0\medspace\text{in}\medspace \mathbb{R}^{N}\setminus \Omega \right\}.
\end{equation*}
Note that for the bounded domain $\Omega$, $	D_{0}^{s,p}(\Omega)=W_{0}^{s,p}(\Omega).$
Next, we recall some weighted H\"{o}lder spaces.
Let the distance function $\;$ $d: \overline{\Omega} \rightarrow \mathbb{R}_+ $ be defined  by 
\begin{equation*}
	\label{dist}d(x):= \text{dist}(x,\partial \Omega),\; x \in \overline{\Omega}.
\end{equation*}
The weighted H\"older type spaces are defined as follows:
\begin{align*}
	& C^0_d(\overline{\Omega}) := \bigg \{ u \in C^0(\overline{\Omega}): 
	\frac{u}{d^s} \text{ admits a continuous extension to } \overline{\Omega}   \bigg\},\\
	& C^{0,\al}_d(\overline{\Omega}) := \bigg \{ u \in C^0(\overline{\Omega}): \frac{u}{d^s} \text{ admits a } \alpha \text{ -H\"older continuous extension to } \overline{\Omega}   \bigg\}
\end{align*}
equipped with the norms
\begin{align*}
	&\|u\|_{C_d^0(\overline{\Omega})}:= \|\frac{u}{d^s}\|_{L^\infty(\Omega)}, \nonumber\\ &\|u\|_{C_d^{0,\alpha}(\overline{\Omega})}:= \|u\|_{C_d^0(\overline{\Omega})}+ \sup_{x, y \in \overline{\Omega},\, x\not = y } \frac{|u(x)/d^s(x)- u(y)/d^s(y)|}{|x-y|^{\alpha}},
\end{align*}
respectively.  The embedding $C^{0,\alpha}_d(\overline{\Omega})\hookrightarrow C^0_d(\overline{\Omega})$ is compact, for all $\alpha\in(0,1).$

\subsection{Notions of solutions}
Next, we specify the two notions of the solutions, viz. weak and the viscosity solutions 
of the following boundary value problem:
\begin{eqnarray}\label{bbp2}   \left\{
	\begin{array}{rl}
		(-\Delta)_p^s u(x) &=h(x,u) \hspace{.50cm} \text{ in }\Omega, \\
		u(x) &=0  \hspace{1.55cm} \text{ on }\Omega^{c}.
	\end{array} 
	\right. 	
\end{eqnarray}
We also consider the equivalence between these two solutions and
refer to \cite{Lindgren1}, and \cite{Barrios20} for the details. 
Let $p\in(1,\infty), $ $h:\Omega\times\mathbb{R}\rightarrow \mathbb{R}$ be a  continuous function which satisfy the growth condition
\begin{equation}
	|h(x,t)|\leq \gamma(|t|)+\phi(x),
\end{equation}
where $\gamma\ge 0$ is continuous and $\phi\in L^{\infty}_{loc}(\Omega).$ We define,
\begin{equation*}
	L_{sp}^{p-1}(\mathbb R^N):= \left \{u \in L^{p-1}_{loc}(\mathbb{R}^{N}):
	\int_{\mathbb{R}^N } \dfrac{|u(x)|^{p-1}}{(1+|x|)^{N+sp}}\,dx<\infty\right\}.
\end{equation*}
\begin{definition} A function $u\in W^{s,p}(\mathbb{R}^{N})\cap L_{sp}^{p-1}(\mathbb R^N)$ is a weak super solution (sub solution) of \eqref{bbp2} if	
	\begin{equation*}
		\int_{\mathbb{R}^N \times \mathbb{R}^N } \dfrac{|u(x)-u(y)|^{p-2}(u(x)-u(y))(\varphi(x)-\varphi(y))}{|x-y|^{N+sp}}  \,dx\,dy 
		 \geq (\leq) 	\int_{\Omega} h(x,u)\varphi \,dx.	
	\end{equation*}
	for every $\varphi \in 	W_{0}^{s,p}(\Omega)$ and $\varphi\geq 0.$	
\end{definition}
%\uttam{Are all sub and super solutions in this space?}
We say that $u$ is a weak solution of \eqref{bbp2} if it is both a weak supersolution and subsolution to the problem.\\
% \textcolor{blue}{We also recall the definition of the viscosity solution of \eqref{bbp2}. Note that in
% the classical $p$-Laplacian problem for the range $1<p<2$, the $p$-Laplacian operator becomes singular when $\nabla u=0$. 
% Analogously, the fractional $p$  Laplace operator is singular for $1<p\leq \frac{2}{2-s}$, in the sense that $(-\Delta)^s_p u(0)$ is defined for the smooth function, $$u(x)=\begin{cases}|x|^2,& x\in B_1\\
% 1,& x\in \mathbb R^N\setminus B_1.
% \end{cases}$$ only for $p> \frac{2}{2-s}$. 
% Thus in order to test the viscosity solution with the test functions of the type $|x|^\beta$ for some $\beta>0$,
% we consider a more restricted class of functions $C_{\beta}^{2}(D)$, as defined below. We refer to \cite{Lindgren1} for more details.}
% So while defining the correct viscosity solution, one has to deal with this issue. The same issue needs % to be addressed in the nonlocal setting as well. For the. To give pointwise sense to the nonlocal operator, we need to work with a more restricted class of functions \cite{Lindgren1}(\textcolor{red}{To Sweta: Please elaborate this line and rewrite this paragraph better})}.\\
\par Next, we define a viscosity sub and super solution of \eqref{bbp2} as given in \cite{Barrios20}. Let $D\subset \Omega$ be an open set; we define
\begin{equation*}
	C_{\beta}^{2}(D):=\left\lbrace u\in C^{2}(D):\sup_{x\in D}\left( \frac{\min\{d_{u}(x),1\}^{\beta-1}}{|\nabla u(x)|}+\frac{|D^{2}u(x)|}{d_{u}(x)^{\beta-2}}\right) <\infty\right\rbrace,
\end{equation*}
where $d_{u}$ is the distance from the set of critical points of $u$ denoted as $N_{u}$, that is
\begin{equation*}
	d_{u}(x):=dist(x,N_{u}),\qquad N_{u}:=\{x\in\Omega:\nabla u(x)=0\}.
\end{equation*}
Motivation for the definition of the space $C_\beta^2(D)$ comes from \cite{Lindgren1}. 
\begin{definition}\cite[Definition 2.2]{Barrios20}
	A function $u$ is a viscosity super-solution (subsolution) of \eqref{bbp2} if 
	\begin{enumerate}[(i)]
		\item $u<+\infty \thickspace(u>-\infty)$ a.e. in $\mathbb{R}^{N}$, $u>-\infty \thickspace (u<+\infty)$ a.e. in $\Omega$.
		\item $u$ is lower (upper) semicontinuous in $\Omega$.
		\item If $\phi\in C^{2}(B(x_{0},r))\cap L_{sp}^{p-1}(\mathbb{R}^{N})$ for some $B(x_{0},r)\subset \Omega$ such that $\phi(x_{0})=u(x_{0}),\phi\leq u (\phi \ge u)$ in $B(x_{0},r),\thickspace\phi=u$ in $\Omega\setminus B(x_{0},r)$ and one of the following conditions hold:
		\begin{enumerate}[(a)]
			\item $p>\frac{2}{2-s}$ or $\nabla \phi(x_{0})\neq 0$,
			\item $1<p\leq \frac{2}{2-s}$, $\nabla \phi(x_{0})= 0$ such that $x_{0}$ is an isolated critical point of $\phi$, and $\phi\in C_{\beta}^{2}(B(x_{0},r))$ for some $\beta>\frac{sp}{p-1}$, 
		\end{enumerate}
		then
			\begin{equation*}
				(-\Delta)_p^s \phi(x_{0}) \ge (\leq) h(x_{0},\phi(x_{0}).
			\end{equation*}
		\item $u_{-}:=\max\{-u,0\}\thickspace (u_{+}:=\max\{u,0\})$ belongs to $L_{sp}^{p-1}(\mathbb{R}^{N})$.	
	\end{enumerate}
	We say that $u$ is a viscosity solution of \eqref{bbp2} if it is both a viscosity supersolution and subsolution to the problem. 	
\end{definition}
\noi \textbf{Notations:} The paper uses various notations, which we will clarify here. Sections 3 and 4 consider the case where $p > \frac{2}{2-s}$. However, for sections 5 and 6, the paper restricts $p$ to the interval $(2,\infty)$ due to the lack of improved weighted regularity results. Unless stated otherwise, $C$, $C_i$, $C_{K,u}$, and other similar notations denote positive generic constants, which may vary in value within the same line. For simplicity, the notation $a^{p-1}$ denotes $a^{p-1}=|a|^{p-2}a$ for any real number $a$.

\section{Pointwise estimates for a barrier function under s-fractional p-Laplace operator} \label{linf}
\noi In this section, we prove some important estimates which are used in subsequent sections to prove the main results on a nonlocal superlinear semipositone problem with subcritical growth.

\noi We begin this section with an important observation that if $u$ is smooth enough, then the 
principal value $P.V.$ in the definition of fractional $p$-Laplacian can be replaced with integral over $\RR^N$ when $p>\frac{2}{2-s}.$ We can also define
\begin{equation} \label{e61}
(-\Delta)_p^s u(x)=P.V. \int_{\RR^N} \frac{(u(x)-u(x+z))^{p-1}+ (u(x)-u(x-z))^{p-1}}{|z|^{N+sp}}\,dz.
\end{equation}
The equivalence of the definitions \eqref{o11} and \eqref{e61}
can be proved using a change of variable.  
\begin{lemma}\label{l62}
	For $\frac{2}{2-s}<p<\infty,$ suppose that $u\in L^\infty(\RR^N) \cap C^{1,1}_{loc}(\Omega),$ then for $x\in \Omega,$
	$$(-\Delta)_p^s u(x)=\int_{\RR^N} \frac{(u(x)-u(x+z))^{p-1}+ (u(x)-u(x-z))^{p-1}}{|z|^{N+sp}}\,dz.$$

\end{lemma}

\begin{proof}Let $p\geq 2$ and $x$ belongs to a compact subset $K$ of $\Omega.$ Using \cite[Lemma 2.11]{t24}, for proper choice of $\gamma\in (0,1),$ we have 
\begin{equation}
    \begin{aligned}
    \int_{B(0,R_{K})} \frac{|(u(x)-u(x+z))^{p-1}+ (u(x)-u(x-z))^{p-1}| }{|z|^{N+sp}}
    \\
    \leq C_{K,u}\int_{B(0,R_{K})}|z|^{\gamma+p-1-N-sp} dz <\infty. 
    \end{aligned}
\end{equation}
	$$ $$
	Since, $u\in L^\infty(\RR^N),$ it can easily be seen that  $$\int_{B^{c}(0,R_{K})} \frac{|(u(x)-u(x+z))^{p-1}+ (u(x)-u(x-z))^{p-1}| }{|z|^{N+sp}} <\infty.$$
	Therefore,   $\displaystyle \frac{|(u(x)-u(x+z))^{p-1}+ (u(x)-u(x-z))^{p-1}| }{|z|^{N+sp}}\in L^1 (\RR^N)$ and hence by dominated convergence theorem we can write, 
	$$   \begin{array}{llr}
		\displaystyle
		\lim_{\epsilon\rightarrow 0} \int_{z\in B^c(0,\epsilon)} \frac{(u(x)-u(x+z))^{p-1}+ (u(x)-u(x-z))^{p-1}}{|z|^{N+sp}}  dz =&&\\[5mm]
		\displaystyle \;\;\;\;\;\;\;\;\;\;\;\;\;\;\;\;\;\;\;\;\; \;\;\;\;\;\;\;\;\;\;\;\;
		\displaystyle \int_{\RR^N} \frac{(u(x)-u(x+z))^{p-1}+ (u(x)-u(x-z))^{p-1}}{|z|^{N+sp}}  dz.
	\end{array}$$
	Now if $p\in (\frac{2}{2-s},2)$ then again using the  second pointwise estimate in \cite[Lemma 2.11]{t24}, 
%	$$|(u(x)-u(x+z))^{p-1}+ (u(x)-u(x-z))^{p-1}| \leq  C_K |z|^{(\gamma+1)(p-1)} $$
and arguing as in the above case, for a suitable choice of $\gamma$, we can prove that 
$$\displaystyle \frac{|(u(x)-u(x+z))^{p-1}+ (u(x)-u(x-z))^{p-1}| }{|z|^{N+ps}}\in L^1 (\RR^N)$$ and thus the conclusion follows as before.
\end{proof}
\subsection{Barrier function under fractional p-Laplacian}\label{Barrier}
\noi 
Our aim in this section is to construct a barrier function and then study the behavior of the barrier function under the fractional 
$p$-Laplacian. Later this result is used to obtain the apriori $L^\infty$ estimate in Theorem \ref{t64}.
To construct the barrier function, we start by considering a smooth perturbation of the distance function $d(x) := \text{dist}(x,\partial\Omega)$, where $x\in\Omega$ is a point in a smooth, bounded domain $\Omega\subset\mathbb{R}^N$ with $N\geq 2$. We define the set $\Omega_\delta= \{x\in \Omega : d(x,\partial\Omega)<\delta\}$ 
for some small $\delta>0$ such that $d(x)$ is well-defined and $C^2$ in $\Omega_\delta$. The barrier function $\xi(x)$ is then constructed as follows:
\begin{equation}\label{e62}\xi(x) =
	\begin{cases}
		d^{\beta}(x)     &\text{if}\thickspace x\in  \Omega_\delta\\
		r(x)  &\text{if}\thickspace	x\in \Omega\setminus \Omega_\delta\\
		0     &\text{if}\thickspace	x\in \Omega^{c}
	\end{cases}
\end{equation}
for $\beta >0$ and a function $r$ such that $\xi$ is positive and $C^{2}$ in $\Omega$
with \begin{equation}\label{e63}
	r(x)\geq d^{\beta}(x) \text{ for } x\in \Omega\setminus \Omega_\delta.
\end{equation}
The proof of the next lemma is inspired by section 3 of \cite{FQ}, where an estimate of a similar type is proved for the fractional Laplacian. However, since the operator we deal with is nonlinear, we have to estimate the integrals differently. We prove the estimate using a different technique that is tailored to the nonlinear operator. 
\begin{lemma}\label{l65} Let $\frac{2}{2-s}<p<\infty$ then there exist $\delta>0$, $C>0$ and  $0<\beta_0<\frac{sp}{p-1}<\infty$ such that for all  $\beta\in(0,\beta_0)$
	\begin{equation*}
		-(-\Delta)_p^s\xi(x)\leq -Cd(x)^{\beta(p-1)-sp}\quad \text{in}\quad \Omega_\delta.		
	\end{equation*}
\end{lemma}
\begin{proof}
	Throughout the proof, we assume that $x\in \Omega_\delta$. 
For simplicity in notation, 	
we write fractional $p$-Laplacian as
\begin{equation}\label{e64}
	-(-\Delta)_p^s \xi(x)=  \int_{\mathbb{R}^N}\frac{\varrho_{+}(\xi,x,y)+\varrho_{-}(\xi,x,y)}{|y|^{N+sp}} \,dy,\quad x\in \mathbb{R}^{N}.
\end{equation}
where
\begin{equation*}
	\varrho_{+}(\xi,x,y)=(\xi(x+y)-\xi(x))^{p-1}	
\end{equation*}
and 
\begin{equation*}
	\varrho_{-}(\xi,x,y)=(\xi(x-y)-\xi(x))^{p-1}	.
\end{equation*}	
Also
\begin{equation*}
    \begin{aligned}
       -(-\Delta)_p^s\xi(x)=\int_{ B(0,\delta)}&\frac{\varrho_{+}(\xi,x,y)+\varrho_{-}(\xi,x,y)}{|y|^{N+sp}}  \,dy
       \\
       &
       + \int_{ B(0,\delta)^c}\frac{\varrho_{+}(\xi,x,y)+\varrho_{-}(\xi,x,y)}{|y|^{N+sp}} \,dy 
    \end{aligned}
\end{equation*}
$$	 $$
Since $\xi$ is an $L^\infty$ function, it is easy to show that there is a constant $C_{1}$ such that 
	\begin{equation}\label{e65}
		\int_{ y\in B^{c}(0,\delta)}\frac{|\varrho_{+}(\xi,x,y)+\varrho_{-}(\xi,x,y)|}{|y|^{N+sp}} \,dy\leq C_{1}.	
	\end{equation}
	
	Now we will estimate the integral over $B(0,\delta)$.  We write
	\begin{equation}\label{e66}
		\int_{ B(0,\delta)}\frac{\varrho_{+}(\xi,x,y)+\varrho_{-}(\xi,x,y)}{|y|^{N+sp}} \,dy = I_{1}(x)+I_{2_{+}}(x)+I_{2_{-}}(x)+I_{3}(x),	
	\end{equation}

	where
	
	\begin{equation}\label{e67}
		\begin{split}
			I_{1}(x)=& \int_{D_{1}}  \frac{-2\xi(x)^{p-1}}{|y|^{N+sp}} \,dy ,\\
			I_{2_{+}}(x)=& \int_{D_{2_{+}}}\frac{\varrho_{+}(\xi,x,y)-\xi(x)^{p-1}}{|y|^{N+sp}} \,dy,\\
			I_{2_{-}}(x)=& \int_{D_{2_{-}}}\frac{\varrho_{-}(\xi,x,y)-\xi(x)^{p-1}}{|y|^{N+sp}} \,dy,\\		
		I_{3}(x)=& \int_{D_{3}}\frac{\varrho_{+}(\xi,x,y)+\varrho_{-}(\xi,x,y)}{|y|^{N+sp}} \,dy	
		\end{split}		
	\end{equation}
	with the following domains of integration.
	\begin{equation}\label{e68}
		\begin{split}
			D_{1}=& \{y\in B(0,\delta): x+y\notin \Omega \quad \text{and}\quad x-y\notin \Omega\},\\
			D_{2_{\pm}}=&\{y\in B(0,\delta): x \pm y\in \Omega \quad \text{and}\quad x\mp y\notin \Omega\},\\
			D_{3}=&\{y\in B(0,\delta): x+y\in \Omega \quad \text{and}\quad x-y\in \Omega\}.
		\end{split}
	\end{equation}
	For notational simplicity we denote $d=d(x)$, whenever there is no confusion. We shall show that each of the integrals $I_i$ is bounded above by $-C d^{\beta(p-1)-sp}.$ The first integral $I_{1}(x)$ can be estimated similar to \cite[Lemma 3.2]{FQ} and we obtain
 \begin{equation}\label{e69}
		I_{1}(x)\leq -C_{2}d^{\beta(p-1)-sp}.	
	\end{equation}
 We will estimate $I_{2_+}(x)$ in detail and $I_{2_-}(x)$ follows in the similar way. We start the estimation of $I_{2_+}(x)$ by making an important observation that $dist(d^{-1}x+z,d^{-1}\partial \Omega)) \leq d^{-1}(|z|+ \inf_{y \in \partial\Om} |x-y|). $ This implies   
	for $0<\beta\leq\beta_{0}<\frac{sp}{p-1},$
	\begin{equation*}
		dist(d^{-1}x+z,d^{-1}\partial \Omega))^{\beta}-1\leq  |z|^{\beta_{0}}	.		\end{equation*}
  The constant $\beta_0$ appearing in the above expression will be chosen later. 
We split the domain $d^{-1}D_{2_{+}}$ into two parts,
\begin{equation*}
		d^{-1}D_{2_{+}}=(d^{-1}D_{2_{+}}\cap B(0,R))\cup (d^{-1}D_{2_{+}}\cap B(0,R)^{c})=:B_{1,R}\cup B_{2,R}.
	\end{equation*}
Then we can write 	$I_{2_{+}}(x)$ as 
\begin{equation}\label{e611}
    \begin{aligned}
    I_{2_{+}}(x)= 	d^{\beta(p-1)-sp}&\int_{B_{1,R}}\frac{(dist(d^{-1}x+z,d^{-1}\partial \Omega)^{\beta}-1)^{p-1}}{|z|^{N+sp}}\,dz\\
    +d^{\beta(p-1)-sp}&\int_{B_{2,R}}\frac{dist(d^{-1}x+z,d^{-1}\partial \Omega)^{\beta}-1)^{p-1}}{|z|^{N+sp}}\,dz\\- d^{\beta(p-1)-sp}&\int_{d^{-1}D_{2_{+}}}\frac{1}{|z|^{N+sp}}\,dz
    \end{aligned}
\end{equation}
The last integral could be easily estimated by $-C_3 d^{\beta(p-1)-sp}.$ We will find a proper upper bound for the first two integrals in terms of $C_3$. Observe that, 
	\begin{equation*}
		\begin{split}
			\int_{B_{2,R}}\frac{dist(d^{-1}x+z,d^{-1}\partial \Omega)^{\beta}-1)^{p-1}}{|z|^{N+sp}}\,dz\leq \int_{B_{2,R}}\frac{|z|^{\beta_{0}(p-1)}}{|z|^{N+sp}}\,dz \;\;\leq \alpha_N R^{\beta_{0}(p-1)-sp}.
		\end{split}	
	\end{equation*}
 where $\alpha_N$ denotes the surface measure of unit sphere in $\mathbb{R}^N.$ Now we shall choose $R$ large enough so that $\alpha_N R^{\beta_{0}(p-1)-sp} = \frac{C_3}{4}.$ We also have 
	\begin{equation*}
		\lim_{\beta \rightarrow 0}\int_{B_{1,R}}\frac{dist(d^{-1}x+z,d^{-1}\partial \Omega)^{\beta}-1)^{p-1}}{|z|^{N+sp}}\,dz=0.
	\end{equation*}
 Now we shall choose $\beta_0$ small enough so that 
 $$\int_{B_{1,R}}\frac{dist(d^{-1}x+z,d^{-1}\partial \Omega)^{\beta}-1)^{p-1}}{|z|^{N+sp}}\,dz <\frac{C_3}{4} \mbox{  for all } 0<\beta<\beta_0.$$
Combining all the above arguments, there exists $ \beta_0<<1 $ and $C_4>0$ such that
\begin{equation}\label{e70}
     \begin{aligned}
        I_{2_+}(x)\leq -C_{4}d^{\beta(p-1)-sp} \thickspace \text{for} \ \beta\in (0,\beta_0) \thickspace 
     \end{aligned}
 \end{equation}
% If $y\in B_{1,R} $	then there exists $r_0>0$ such that $|y|>r_0$ because otherwise it would imply that $x\in \partial\Om$. So we could estimate the preceding integral for every $\beta\in(0,\beta_0)$. This completes the estimation of $I_{2_{+}}(x)$. 
Finally, we will estimate $I_{3}(x)$ 
 by splitting the domain of its integration $$D_3= B(0,\epsilon d)\cup (D_{3}\setminus B(0,\epsilon d))=:B_{1}\cup B_{2}.$$
 We rewrite the $I_3(x)$ as
 \begin{equation}\label{e313}
     \begin{aligned}
        I_3(x) =\int_{B_{2}}\frac{\varrho_{+}(\xi,x,y)+\varrho_{-}(\xi,x,y)}{|y|^{N+sp}} \,dy	+ \int_{B_{1}}& \frac{\varrho_{+}(\xi,x,y)}{|y|^{N+sp}} \,dy 
        \\
        &+ \int_{B_{1}}\frac{\varrho_{-}(\xi,x,y)}{|y|^{N+sp}} \,dy
     \end{aligned}
 \end{equation}
On $B_{2}$ the estimation of $I_{3}(x)$ is similar to $I_{2_{+}}(x)$ hence 
\begin{equation}\label{e314}
    \begin{aligned}
     \int_{B_{2}}\frac{\varrho_{+}(\xi,x,y)+\varrho_{-}(\xi,x,y)}{|y|^{N+sp}} \,dy \leq  -C_{4}d^{\beta(p-1)-sp}  
    \end{aligned}
\end{equation}
 We only estimate 
	%$\varrho_{+}(\xi,x,y)$
	$\displaystyle\int_{B_1}\frac{\varrho_{+}(\xi,x,y)}{|y|^{N+sp}} \,dy$ , as the estimation of 
$\displaystyle \int_{B_1}\frac{\varrho_{-}(\xi,x,y)}{|y|^{N+sp}} \,dy$ follows using the similar arguments. %In conclusion we only estimate
%\begin{equation*}
	%	\left| \int_{B(0,\epsilon d)}\frac{(\xi(x+y)-\xi(x))^{p-1}}{|y|^{N+sp}}\,dy\right|
		%	\end{equation*}
	%for some $\epsilon>0$ 
 
% In the light of \cite[Lemma 3.2]{FQ}, we only need to estimate two integrals, namely $I_{2_+}(x)$ and 
% 	Estimation of $I_{2_-}(x)$ follows easily from the estimation of $I_{2_+}(x)$. Observe that $	D_{3}=B(0,\epsilon d)\cup (D_{3}\setminus B(0,\epsilon d))=:B_{1}\cup B_{2}$. 
Since $\xi$ is a $C^2$ function, using Taylor's expansion, we have
 \begin{equation*}
		\xi(x+y)=\xi(x)+\nabla\xi(x)\cdot y+y^\top\cdot  D^{2}\xi(\alpha)\cdot y \quad\text{for some}\quad \alpha= x+\theta y,\;\theta\in(0,1).
	\end{equation*}
	We denote 
	\begin{equation*} 
		l(x+y)=\xi(x)+\nabla\xi(x)\cdot y,
	\end{equation*}
	so that we have
	\begin{equation*}
		\xi(x+y)-l(x+y)=y^\top\cdot  D^{2}\xi(\alpha)\cdot y.	
	\end{equation*}
	Similarly, we have
	\begin{equation*}
	    \begin{aligned}
	    \xi(x)=\xi(x+y)+\nabla\xi(x+y)\cdot (-y)+(-y)^\top\cdot&  D^{2}\xi(\gamma)\cdot (-y)
	    \\
	    \text{for some}\;\; &\gamma= x+y-\theta y,\theta\in(0,1).
	    \end{aligned}
	\end{equation*}
	\begin{equation*}
		 \quad	
	\end{equation*}
	We denote
	\begin{equation*}
		l(x)=\xi(x+y)+\nabla\xi(x+y)\cdot (-y),	
	\end{equation*}
	so we have
	\begin{equation*}
		\xi(x)-l(x)=(-y)^\top\cdot  D^{2}\xi(\gamma)\cdot (-y)
	\end{equation*}
	and
	\begin{equation*}
		l(x+y)-l(x)=\nabla\xi(x)\cdot y.	
	\end{equation*}
	Let	$g(t):=|t|^{p-2}t$ and using \cite[Lemma 3.1 and 3.4]{Lindgren1}, we have
	\begin{equation}\label{eqn312}
		\begin{split}
			&\left|\int_{B(0,\epsilon d)}\frac{(\xi(x+y)-\xi(x))^{p-1}}{|y|^{N+sp}}\,dy\right|\\
			&\leq\int_{B(0,\epsilon d)}\frac{|g(\xi(x+y)-\xi(x))-g(l(x+y)-l(x))|}{|y|^{N+sp}}\,dy\\
			&\leq C_{5}\int_{B(0,\epsilon d)}\frac{(|l(x+y)-l(x)|+|\xi(x)-l(x)|)^{p-2}|\xi(x)-l(x)|}{|y|^{N+sp}}\,dy\\
			&\leq C_{5} \int_{B(0,\epsilon d)}\frac{(|\nabla\xi(x)\cdot y|+|(-y)^\top\cdot  D^{2}\xi(\gamma)\cdot (-y)|)^{p-2}|(-y)^\top\cdot  D^{2}\xi(\gamma)\cdot (-y)|}{|y|^{N+sp}}\,dy.
		\end{split}
	\end{equation}
		Since $x\in \Omega_\delta$, $\xi(x)=d^{\beta}(x)$. Then
	\begin{equation*}
		\begin{split}
			\nabla\xi(x)&=\beta d^{\beta -1}\nabla d(x) \mbox{ and } 
			\frac{\partial^{2}\xi(x)}{\partial x_{i}\partial x_{j}}=\beta (\beta-1)d^{\beta-2}A_{ij}(x) \text{ for } 1\leq i,j\leq N,
		\end{split}	
	\end{equation*}
	where
	\begin{equation*}
		A_{ij}(x)=\frac{\partial d(x)}{\partial x_{i}}\frac{\partial d(x)}{\partial x_{j}}+\frac{d(x)}{\beta -1}\frac{\partial^{2}d(x)}{\partial x_{i}\partial x_{j}}.
	\end{equation*} Then for $\gamma=x+y-\theta y, \thinspace \theta\in(0,1)$ and $x\in \Omega_\delta, y\in B(0,\epsilon d),$  we denote
	\begin{equation*}
	 M_{i,j}:=	\sup_{\gamma}\beta(\beta-1)A_{ij}(\gamma), 
	\end{equation*}
	Thus we have $|D^{2}\xi(\gamma)|\leq M d^{\beta-2}$ for some $M>0$.
	 Hence from \eqref{eqn312} and with a change of variable $y=d(x)z$ we have, 
	 \begin{equation*}
		\begin{split}
			&\left|\int_{B(0,\epsilon d)}\frac{(\xi(x+y)-\xi(x))^{p-1}}{|y|^{N+sp}}\,dy\right|\\
			&\leq C_{6}d^{\beta (p-1)-sp}\int_{B(0,\epsilon)}\frac{(|\beta \nabla d(x)\cdot z|+M|z|^{2})^{p-2}M|z|^{2}}{|z|^{N+sp}}\,dz\\
			&\leq C_{6}d^{\beta (p-1)-sp}\int_{0}^{\epsilon} \int_{S^{N}}\left(|\beta \nabla d(x)\cdot \omega|r+Mr^{2}\right)^{p-2}Mr^{1-sp}\,d\omega\,dr\\
   	&\leq C_{6}d^{\beta (p-1)-sp}\int_{0}^{\epsilon} \int_{S^{N}}\left(\frac{|\beta \nabla d(x)\cdot \omega|}{|\beta \nabla d(x)|}+\frac{Mr}{|\beta \nabla d(x)|}\right)^{p-2}
Mr^{p-1-sp}|\beta \nabla d(x)|^{p-2}\,d\omega\,dr\\
			&\leq C_{6}Md^{\beta (p-1)-sp}\int_{0}^{\epsilon} \left(1+\frac{Mr}{|\beta \nabla d(x)|}\right)^{p-2}r^{p-1-sp}|\beta \nabla d(x)|^{p-2}\,dr 	
		\end{split}	
	\end{equation*}
	   %	\\
   	% &\hspace{8.6 cm}
 The last inequality here is obtained by applying \cite[Lemma 3.5]{Lindgren1}.  We can now directly use the estimates exactly as in \cite[Lemma 3.6]{Lindgren1} and for the case $p\ge 2$ we get 
 \begin{equation}
     \begin{aligned}\label{e613}
			\bigg|\int_{B(0,\epsilon d)}&\frac{(\xi(x+y)-\xi(x))^{p-1}}{|y|^{N+sp}}\,dy\bigg|\\
			&\hspace{2.2cm}\leq C_{5}Md^{\beta (p-1)-sp}\left(|\beta \nabla d(x)|^{p-2}\epsilon^{p-sp}+M^{p-2}\epsilon^{p-2+p(1-s)}\right) \\
   &\hspace{8.6cm} \leq C_6 d^{\beta (p-1)-sp}\e^{p-sp}
		\end{aligned}
 \end{equation}
% 	 	\begin{align}\label{e613}
% 			&\left|\int_{B(0,\epsilon d)}\frac{(\xi(x+y)-\xi(x))^{p-1}}{|y|^{N+sp}}\,dy\right|\nonumber\\
% 			&\leq C_{5}Md^{\beta (p-1)-sp}\left(|\beta \nabla d(x)|^{p-2}\epsilon^{p-sp}+M^{p-2}\epsilon^{p-2+p(1-s)}\right)\nonumber \\
%   & \leq C_6 d^{\beta (p-1)-sp}\e^{p-sp}\nonumber
% 		\end{align}
		and for $\frac{2}{2-s}<p<2,$
\begin{equation}\label{ee614}
   \left|\int_{B(0,\epsilon d)}\frac{(\xi(x+y)-\xi(x))^{p-1}}{|y|^{N+sp}}\,dy\right|\leq C_{6}d^{\beta (p-1)-sp}\epsilon^{p-2+p(1-s)}.\nonumber
\end{equation}	
In both the cases choosing $\e$ small enough we have 
\begin{equation}\label{e316}
    \begin{aligned}
\int_{B_{1}}\frac{\varrho_{+}(\xi,x,y)}{|y|^{N+sp}} \,dy + \int_{B_{1}}\frac{\varrho_{-}(\xi,x,y)}{|y|^{N+sp}} \,dy \leq \frac{C_4}{2} d^{\beta (p-1)-sp}        
    \end{aligned}
\end{equation}
Using \eqref{e313} \eqref{e314} and \eqref{e316} we conclude that 
\begin{equation}
    \begin{aligned}
        I_3(x)\leq -C_7 d^{\beta(p-1)-sp}
    \end{aligned}
\end{equation}
This concludes the proof of the Lemma.
% Now using these estimates, we get our desired results for $\frac{2}{2-s}< p$ with the help of ideas of \cite{FQ}.

% So, for $p\ge 2$, using \eqref{e65}, \eqref{e69}, \eqref{e70}, \eqref{e613}, we have
% \begin{align*}
% 		-(-\Delta)_p^s\xi(x)\leq &d(x)^{\beta(p-1)-sp} \left( C_{1}d(x)^{sp-\beta(p-1)}-C_{2}-C_{4}\right)\\
% 		&+C_{6}Md(x)^{\beta(p-1)-sp}\left(|\beta \nabla d(x)|^{p-2}\epsilon^{p-sp}+M^{p-2}\epsilon^{p-2+p(1-s)}\right). 		
% 	\end{align*}
% Thus, there exists $C>0$ such that
% 	\begin{equation*}
% 	-(-\Delta)_p^s\xi(x)\leq -Cd(x)^{\beta(p-1)-sp}\quad \text{in}\quad \Omega_\delta.		
% \end{equation*}
% Similarly, for $\frac{2}{2-s}<p<2$, using \eqref{e65}, \eqref{e69}, \eqref{e70}, \eqref{ee614}, we get the desired result.
\end{proof}

\noi We define for $x_0\in\partial\Omega,\thickspace\tau>0$
\begin{equation*}
	\Omega^{\tau}_{x_0}:=\left\lbrace x\in \mathbb{R}^{N}:x_0+\tau x\in \Omega\right\rbrace .
\end{equation*}
\noi Also for the function $\xi$ defined in \eqref{e62}, we set
%\begin{equation*}
${\xi^{\tau}_{x_0}}(x):=\xi(x_0+\tau x)$
%\end{equation*}
and define
\begin{equation}\label{e616}
	d_{\tau}(x):=dist(x,\partial\Omega^{\tau}_{x_0})=\tau^{-1}dist(x_0+\tau x,\partial\Omega).
\end{equation}

\begin{lemma}\label{l66}
	Let $s\in(0,1)$ and $\frac{2}{2-s}< p<\infty$. Also let $\beta=sp-\theta$ in \eqref{e62} for $\theta\in\left(  sp-\beta_0,sp\right)$ for some $\beta_0 \in (0,\frac{sp}{p-1})$. Then there exist $c_{0},\delta >0$ such that 
	\begin{equation}\label{e617}
		(-\Delta)_{p}^s \xi^{\tau}_{x_0}\ge c_{0}d^{sp^{2}-2sp-\theta(p-1)}_{\tau}\quad \text{ in }\thickspace (\Omega^{\tau}_{x_0})_{\delta},\thickspace 0<\tau<1.
	\end{equation}	 
	Moreover, if $u\in L^{\infty}(\RR^N)$ satisfies the following in viscosity sense  for $c_1>0$
	\begin{eqnarray}  \left\{
		\begin{array}{rl}
			(-\Delta)_{p}^s u &\leq c_{1}d^{sp^{2}-2sp-\theta(p-1)}_\tau \quad \text{ in }\Omega^{\tau}_{x_0}, \\
			u &=0  \quad \hspace{2.6cm}\text{ on }(\Omega^{\tau}_{x_0})^{c}.
		\end{array} 
		\right. 
	\end{eqnarray}
	then 
	\begin{equation}\label{e618}
		u(x)\leq c_{2}\left( c_{1}+\|u\|_{L^{\infty}(\Omega^{\tau}_{x_0})}\right) d^{sp-\theta}_{\tau}, \thickspace x\in (\Omega^{\tau}_{x_0})_{\delta}
	\end{equation}
	for some $c_{2}>0$ is only depending on $s,\delta,\theta,p$ and $c_{0}$. 
\end{lemma}
\begin{proof}
    Observe that for
	$x\in(\Omega^{\tau}_{x_0})_\delta$ we have
	$x_0+\tau x\in(\Omega)_{\tau\delta}\subset(\Omega)_{\delta}$ as $\tau<1$.
		The rest of the argument for establishing \eqref{e617} follows using similar calculations as in the proof of the Lemma \ref{l65} by taking $\beta=sp-\theta$ and translating $\xi(x)$ to $\xi^\tau_{x_0}(x)$ and 
	$(\Omega)_{\delta}$ to $(\Omega^{\tau}_{x_0})_\delta$.\\
Next, to obtain \eqref{e618}, we define $v=R\xi^{\tau}_{x_0}$ where
\begin{equation*}
    \begin{aligned}
    R=\left( \frac{c_{1}}{c_{0}}\right)^{1/(p-1)}+(\tau\delta)^{\theta-sp}\|u\|_{L^{\infty}(\Omega^{\tau}_{x_0})}.
    \end{aligned}
\end{equation*}
and observe that 
	\begin{eqnarray}\label{e620}  \left\{
		\begin{array}{rl}
				(-\Delta)_{p}^sv &\ge (-\Delta)_{p}^su   \quad \text{ in }(\Omega^{\tau}_{x_0})_{\delta}, \\
			u = v &=0  \quad \hspace{1.2cm}\text{ on }(\Omega^{\tau}_{x_0})^{c}\\
			v &\ge u \hspace{1.6cm}\text{ in }\Omega^{\tau}_{x_0}\setminus (\Omega^{\tau}_{x_0})_{\delta}.
		\end{array} 
		\right. 
	\end{eqnarray}
		Using the comparison principle  we can get 
	\begin{equation*}
		u(x)\leq c_{2}\left( (c_{1})^{1/(p-1)}+\|u\|_{L^{\infty}(\Omega^{\tau}_{x_0})}\right) d^{sp-\theta}_{\tau}\quad \text{ in }\thickspace (\Omega^{\tau}_{x_0})_{\delta}
	\end{equation*}
	where $c_{2}=\left( (\tau\delta)^{\theta-sp}+(c_{0})^{-1/(p-1)}\right)$.
\end{proof}

\section{An apriori uniform boundedness estimate for subcritical problem}\label{linf1}
\noi In this section we give the uniform $L^\infty$ estimates for the solutions of the semipositone subcritical problem. 
\noi We recall the local H\"{o}lder estimates for viscosity solutions from  \cite[Theorem 1]{HL1}.
\begin{theorem}\label{t62} 
For $1<p<\infty$, assume $f\in C(B_{2}(0))\cap L^{\infty}(B_{2}(0))$ and let $u\in L^{\infty}(\mathbb{R}^{N})$ be viscosity solution of 
	\begin{equation*}
		(-\Delta)_{p}^s u =f\quad \text{ in } B_{2}(0)
		.
	\end{equation*}
	Then $u$ is H\"older continuous in $B_{1}(0)$ and in particular there exist $\alpha\in(0,1)$ and $c$ depending on $s,p$ such that 
	\begin{equation}\label{e615}
		\|u\|_{C^{\alpha}(B_{1}(0))}\leq c\left( \|u\|_{L^{\infty}(\mathbb{R}^{N})}+\|f\|^{\frac{1}{p-1}}_{L^{\infty}(B_{2}(0))}\right).
	\end{equation}
\end{theorem}

\begin{theorem}\label{t64}
Assume that, $0<s<1,$ and $\frac{2}{2-s}<p< \infty,\thinspace	p-1<q<\pst -1$, and $g\in C(\overline{\Omega}\times\mathbb{R})$ satisfies that $|g(x,z)|\leq c|z|^{r},\thinspace x\in \Omega,\thinspace z\in \mathbb{R}$ where $c>0$ and $0<r<q$. Let $u$ be a positive viscosity solution to the problem
	\begin{eqnarray} \label{p63}  \left\{
		\begin{array}{rl}
			(-\Delta)_p^s u(x) &=u^{q}+g(x,u) \hspace{.50cm} \text{ in }\Omega, \\
			u(x) &=0  \hspace{2.42cm} \text{ on }\Omega^{c}.
		\end{array} 
		\right. 	
	\end{eqnarray}
	Then there exists a constant $C>0$ such that 
	\begin{equation*}
		\|u\|_{L^{\infty}(\Omega)}\leq C.
	\end{equation*}
\end{theorem}
\begin{proof}
We begin our proof by assuming the existence of a sequence of positive solutions $\{u_{k}\}$ of problem \eqref{p63} such that $M_{k}=\|u_{k}\|_{L^{\infty}(\Omega)}\rightarrow \infty$.
% On the contrary, there exists a sequence of positive solutions $\{u_{k}\}$ of problem \eqref{p63} such that $M_{k}=\|u_{k}\|_{L^{\infty}(\Omega)}\rightarrow \infty$.    % Similar to \cite[Theorem 8]{BLG}, we also start our proof by assuming the existence of a sequence of positive solutions $\{u_{k}\}$ of problem \eqref{p63} such that $M_{k}=\|u_{k}\|_{L^{\infty}(\Omega)}\rightarrow \infty$. 
For $x_{k}\in \Omega$ such that $u_{k}(x_{k})=M_{k}$, we define the functions \begin{equation*}
		v_{k}(y)=\frac{u_{k}(x_{k}+\mu_{k}y)}{M_{k}},\thickspace y\in\Omega^{k},
	\end{equation*}
	where
	\begin{equation*}
	    \begin{aligned}
	    	\Omega^{k}:=\left\lbrace y\in \mathbb{R}^{N}:x_{k}+\mu_{k}y\in \Omega\right\rbrace ,
		\\
		\mu_{k}= M^{\frac{-(q-(p-1))}{sp}}_{k}\in \mathbb{R} \longrightarrow 0,
	    \end{aligned}
	\end{equation*}
and functions $v_{k}$ satisfies $0<v_{k}\leq 1$ and $v_{k}(0)=1$. Next we have
\begin{equation}\label{e621}
		(-\Delta)_p^s 	v_{k}(y)= v^{q}_{k}+h_{k}\quad \text{in}\quad 	\Omega^{k}.
	\end{equation}
	where $|h_{k}|\leq C M^{r-q}_{k}$ and $h_{k}\in C(\Omega^{k})$. Next passing to subsequences, there are two cases, either $d(x_{k})\mu^{-1}_{k}\rightarrow\infty$ or $d(x_{k})\mu^{-1}_{k}\rightarrow \Tilde{d}$ for some non-negative constant $\Tilde{d}$. The first case implies $\Omega_{k}\rightarrow \mathbb{R}^{N}$. The right hand side in \eqref{e621} is uniformly bounded, now as in \cite{BLG} using estimates in \eqref{e615} with an application of Arzela-Ascoli's theorem and a diagonal argument we get $v_k\rightarrow v$ locally uniformly in $\mathbb{R}^N$ up to a subsequence. Thanks to \cite[Corolary 4.7]{LLCS}, the limiting function $v$ is indeed a nontrivial positive viscosity solution to the problem $(-\Delta)^s_p v= v^p$ in $\mathbb{R}^N$. Then by Theorem \ref{t62},  we get $v\in C^{\alpha}(\mathbb{R}^N)$ for some $\alpha \in (0,1)$ which contradicts the nonexistence hypothesis $(\mathcal{NA}).$ 
% 	Then by Theorem 3.4,  up to a subsequence $v_k\rightarrow v$ locally uniformly in $\mathbb{R}^N.$  and dh\\
\\
Next we shall consider the case $d(x_{k})\mu^{-1}_{k}\rightarrow \Tilde{d}\ge 0$ which implies $x_k\rightarrow x_0$ for some $x_0\in \pa\Omega.$ Without loss of generality, we assume outward normal at $x_0,$ i.e $\nu (x_{0}) =-e^{N}$. Define
	\begin{equation*}
		w_{k}(y)=\frac{u_{k}(\zeta_{k}+\mu_{k}y)}{M_{k}}\quad y\in D^{k},
	\end{equation*} 
	where $\zeta_{k}\in \partial\Omega$ is the projection of $x_{k}$ on $\partial\Omega$ and 
	\begin{equation}
		D^{k}:=\{y\in \mathbb{R}^{N}:\zeta_{k}+\mu_{k}y\in \Omega\}.\nonumber
	\end{equation}
	Observe that 
	\begin{equation}\label{e623}
		0\in \partial D^{k},
	\end{equation}
	and
	\begin{equation*}
		D^{k}\rightarrow\mathbb{R}^{N}_{+}=\{y\in \mathbb{R}^N:y_{N}>0\}\thickspace \text{as}\medspace k\rightarrow +\infty.
	\end{equation*}
	It is easy to show that $w_{k}$ satisfies \eqref{e621} in $D^{k}$ with a different function $h_{k}$, but with the same bounds. Setting
	\begin{equation*}
		y_{k}:=\frac{x_{k}-\zeta_{k}}{\mu_{k}},
	\end{equation*}
	so that $|y_{k}|=d(x_{k})\mu^{-1}_{k}$, we observe that $w_{k}(y_{k})=1$. Since $|y_k|\rightarrow \Tilde{d},$ by passing to further  subsequence $y_{k}\rightarrow y_{0}$; $|y_0|=\Tilde{d}\geq 0$. Next we claim the following:  \\
 \underline{Claim:} 
 $y_{0}$ is in the interior of half-space $\mathbb{R}^{N}_{+}$. For this it is sufficient to show that 
	\begin{equation}\label{e6000}
		\Tilde{d}=\lim_{k\rightarrow \infty}d(x_{k})\mu^{-1}_{k}>0.
	\end{equation}
 % Next by passing to further  subsequence $y_{k}\rightarrow y_{0}$; we claim that 
Now our estimate, namely Lemma \ref{l66}, plays a crucial role in establishing this claim. We observe that by \eqref{e621}, and since $r<q$, we get 
	\begin{equation*}
		(-\Delta)_p^s w_{k}	\leq C\leq C_{1}d^{sp^{2}-2sp-\theta(p-1)}_{k}\thickspace \text{in}\medspace D^{k},
	\end{equation*}
  where $d_{k}(y)=dist(y,\partial D^{k})$ and for a fixed $\theta\in\left(  sp-\beta_0,sp\right)$ as given in Lemma \ref{l66}.   If $d_{k}(y)<\delta$, by \eqref{e618}  we have $w_{k}(y)\leq C_{0}d_{k}(y)^{sp-\theta}$ for some $C_0>0.$ Now from \eqref{e623} and using the definition of $d_k(y_k)$ clearly $|y_{k}|\geq d_{k}(y_{k}).$ Combining all these we get the following if $d_k(y_k)<\delta$ $$1=w_k(y_k) \leq C_{0}d_{k}(y_k)^{sp-\theta} \leq  C_{0}|y_{k}|^{sp-\theta}.$$ This asserts that $|y_{k}|$ is bounded below by a positive constant, thus the claim follows. \\
Next, similar to \cite[Page no 208]{BLG} we can use the estimates in \eqref{e615} with an application of Arzela-Ascoli's theorem and a diagonal argument to conclude $w_{k}\rightarrow w$ uniformly on a compact set of $\mathbb{R}^{N}_{+}$, with $0\leq w\leq 1.$ Since $y_k\rightarrow y_0,$ for some $y_0$ is in the interior of $\mathbb{R}^N_+,$ the uniform convergence of $w_k$ in the compact subsets of $\mathbb{R}^N_+$ implies that $w(y_{0})=1.$ and $w(y)\leq Cy^{sp-\theta}_{N}$ for $y_{N}<\delta.$ Thus $w\in C(\mathbb{R}^{N}_+)$ is non-negative, bounded solution of 
	\begin{eqnarray*}  \left\{
		\begin{array}{rl}
			(-\Delta)_{p}^s w &=w^{q}\quad \text{ in }\mathbb{R}^{N}_{+}, \\
			w &=0  \hspace{.7cm} \text{ in }\mathbb{R}^{N}\setminus \mathbb{R}^{N}_{+}. 
		\end{array} 
		\right. 
	\end{eqnarray*} 
	Again using Theorem \ref{t62} we get $w\in C^{\alpha}(\mathbb{R}^N)$ for some $\alpha \in (0,1)$. Since $w(y_0)=1$, the strong maximum principle implies $w>0$. 
	So we conclude our theorem by contradiction to the nonexistence assumption $(\mathcal{NA})$.
\end{proof}
\begin{remark}
     Results discussed in section \ref{linf} and section \ref{linf1} do not use the convexity of the domain $\Omega.$ 
\end{remark}
\section{Nonlocal superlinear semipositone problem with subcritical growth}\label{deg}
 \noi In this section  we look for a positive solution to following the Dirichlet boundary value problem 
 \begin{eqnarray}\label{p61} (P^\mu)  \left\{
 	\begin{array}{rl}
 		(-\Delta)_p^s u &=\mu( u^{r}-1) \text{ in }\Omega, \\
 		u &>0  \hspace{1.40cm} \text{ in }\Omega,\\
 		u &=0 \hspace{1.40cm} \text{ on }\Omega^{c},
 	\end{array} 
 	\right. 
 \end{eqnarray}
 where $p-1<r< p^{*}_{s}-1$.  Using the substitution $w=\gamma u$ where $\gamma^{r+1-p}=\mu$, we see that \eqref{p61} is equivalent to the nonlocal problem
\begin{eqnarray}\label{p64}  \left\{
	\begin{array}{rl}
		(-\Delta)_p^s w &= w^{r}-\gamma^r ~~ \text{ in }\Omega, \\
		w &>0  \hspace{1.40cm} \text{ in }\Omega,\\
		w &=0 \hspace{1.40cm} \text{ on }\Omega^{c}.
	\end{array} 
	\right. 
\end{eqnarray}
We use this observation to study the equivalent problem (\ref{p64}).
\begin{definition}
	We define the operator $K: C(\overline{\Om})\rightarrow C(\overline{\Omega})\cap W_0^{s,p}(\Om)$ as $K(f)=u$ where $u$ is the unique weak solution of $\fpl u= f$ in $\Omega$ and $u=0$ on $\Om^c.$
\end{definition} 
The weak solution $u$ can be obtained as the minimizer of the associated functional in $W^{s,p}_{0}(\Omega)$. Using \cite[Theorem 1.1]{t24}, $u\in C^\al(\overline{\Omega})$ for some  $\al \in (0,s]$ and thus the map $K$ is well defined. From Theorem \cite[Theorem 1.3]{Barrios20}, we also infer that the weak solution $u$ is in fact a viscosity solution. Now, if we set
\begin{equation}\label{e624}
	F(\mu,u)=\mu(|u|^r-1)
\end{equation}
then finding a weak solution to the nonlinear problem 
\begin{eqnarray}\label{p65}  \left\{
	\begin{array}{rl}
		(-\Delta)_p^s u &=  F(\mu,u)~  \text{ in }\Omega, \\
		u &>0  \hspace{1.20cm} \text{ in }\Omega,\\
		u &=0 \hspace{1.20cm} \text{ on }\Omega^{c}.
	\end{array} 
	\right. 
\end{eqnarray}
is equivalent to finding a fixed point for the map $K( F(\mu, u)).$ By the regularity results for subcritical problem (see \cite{IASLKM}), the weak solution of (\ref{p65}) is infact a viscosity solution. Using the rescaling argument once again, we see that  $u=K( F(\mu,u))$ iff $w= K( \tilde{F}(\gamma,w))$ where $\tilde{F}(\ga, w)= |w|^{r}-\gamma^r$ i.e. 
\begin{eqnarray}\label{p66}  \left\{
	\begin{array}{rl}
		(-\Delta)_p^s w &=   |w|^r-\gamma^r  \text{ in }\Omega, \\
		%	w &>0  \hspace{1.40cm} \text{ in }\Omega,\\
		w &=0 \hspace{1.20cm} \text{ on }\Omega^{c}.
	\end{array} 
	\right. 
\end{eqnarray}
 We shall denote 
\begin{equation}\label{e625}
	S(\gamma,w)=w-K( \tilde{F}(\gamma,w) )\mbox{ for } 0\leq \gamma<\infty.
\end{equation}
For $\gamma=0, $ the map $S(0,w)$ is denoted as $S_0(w).$ The solutions of $S_0(w)=0$ are nothing but the solutions of 

\begin{eqnarray}\label{p67}  \left\{
	\begin{array}{rl}
		(-\Delta)_p^s w &=   |w|^r  \text{ in }\Omega, \\
		%	w &>0  \hspace{1.40cm} \text{ in }\Omega,\\
		w &=0 \hspace{.40cm} \text{ on }\Omega^{c}.
	\end{array} 
	\right. 
\end{eqnarray}
\\
%\noi Next we have the following lemma which follows
%using the strict convexity of the domain and the monotonicity of the solutions 
%given as \cite[Theorem 1.1]{WWLZ}.
%\begin{lemma}\label{l68}
%	Let $\Omega$ be a strictly convex bounded smooth domain, and
%	define $\Omega_\delta=\{x \in\Omega : dist (x, \partial\Omega) > \delta\}$ for $\delta> 0$.
%	Then the following holds for a weak solution $u\in C^{0,\al}_d(\overline{\Om})$ for some $\alpha \in(0,s]$ of the problem
%	\eqref{p67}
%	\begin{align}\label{e626}
%		&\text{ there exist }\gamma, \epsilon > 0, \text{ depending only on }\Omega,\text{  such that }\nonumber\\
%		&\forall~ x \in\Omega\setminus\Omega_\epsilon \text{ there is a part of a cone }I_x
%		\text{ with }\nonumber\\
%		&~~(i) u(\xi) \geq u(x)~ \forall~ \xi\in I_x ,\nonumber\\
%		&~~ (ii) I_x \subset \Omega_{\frac{\epsilon}{ 2}},\nonumber\\
%		&~~(iii) meas (I_x )\geq\gamma.
%	\end{align}	
%\end{lemma}
\noi Next we state \cite[Proposition 2.1 and Remark 2.1]{DDPLNR} here  which will be used in the proof of Lemma \ref{l67}.
\begin{proposition}\label{pr61}
	Let $C$ be a cone in a Banach space $X$ and $\Phi: C\rightarrow C$ be a compact map such that $\Phi(0)=0.$ Assume there exists $0<R_1<R_2$ such that 
	\begin{itemize}
		\item [(i)] $x\neq t \Phi(x)$ for all $t\in [0,1]$ and $\|x\|_{X}=R_1$.
		\item [(ii)] There exists a map $T: \overline{B}_{R_2}\times [0,\infty) \rightarrow C$ such that $T(x,0)=\Phi(x)$ for all $\|x\|_X=R_2, T(x,t)\neq x$ for $\|x\|_{X}=R_2$ and $0\leq t<\infty$ and $T(x,t)=x$ has no solution for $x\in \overline{B}_{R_2},\, t\geq t_0.$
	\end{itemize}
	If we denote $U=\{x\in C: R_1< \|x\|_{X}<R_2\}$ and $B_\rho=\{x\in C: \|x\|_{X}<\rho\}$, then 
	$deg(I-\Phi,B_{R_2},0)=0, deg(I-\Phi,B_{R_1},0)=1$ and $deg(I-\Phi,U,0)=-1.$
\end{proposition}

\begin{lemma}\label{l67}
	There exists $0<R_1<R_2$ such that $S_0(w)\neq 0$ for all $w$ belonging to the set $\{w: \|w\|_{\infty}=R_1 \; \text{or} \;R_2\}$  and
	$deg(S_0, B_{R_2}\setminus\overline{B}_{R_1},0)=-1.$
\end{lemma}
\begin{proof}
Let us consider $X=C(\overline{\Om}), $ and  $C=\{u\in X: u(x)\geq 0\}$. We define a map $\Phi: C\rightarrow C$ by setting $\Phi(\cdot):= K(\tilde{F}(0,\cdot))$. It is straightforward to show that $\Phi$ is a compact map.
 Suppose that $t \Phi(u)=u$ for some $t\in [0,1]$ then,  
	\begin{eqnarray}\label{p68}  \left\{
		\begin{array}{rl}
			(-\Delta)_p^s u&=   t^{p-1}|u|^r ~~ \text{ in }\Omega, \\
			%	w &>0  \hspace{1.40cm} \text{ in }\Omega,\\
			u &=0 \hspace{1.40cm} \text{ on }\Omega^{c}.
		\end{array} 
		\right. 
	\end{eqnarray}
		Using the variational characterization of the principal eigenvalue and the $L^\infty$ regularity of the weak solution of \eqref{p68} 
	$$\la_{1}^p \int_{\Om} |u|^p \leq \|u\|_{W_0^{s,p}(\Om)} =t^{p-1}\int_{\Omega} |u|^{r+1}\leq \|u\|_{L^{\infty}(\Om)}^{r+1-p}\int_{\Om} |u|^p.$$
	which imples, $\|u\|_{L^{\infty}(\Om)}\geq c(s,p,r).$ Now by choosing $R_1$ small enough we have condition $(i)$ listed in  Proposition \ref{pr61}.\\
	Next we define $T(u,t)=K (\tilde{F}(0,(|u|+t))). $ Then $T(u,0)=\Phi(u)$ and we will verify two more conditions of the map $T$, viz.,
	\begin{itemize}
		\item [(a)] $T(u,t)\neq u$ for all $\|u\|_{L^{\infty}(\Om)}=R_2$ and $0\leq t<\infty.$
		\item[(b)]  $T(u,t)=u$ has no solution for $u\in \overline{B}_{R_2}$ and $t\geq t_0.$
	\end{itemize}
	To verify (b), we infact claim that $T(u,t)=u$ has no solution if $t\geq t_0.$ Suppose that for any arbitrary $t$ there exists a solution $u_t\in C(\overline{\Om})$ of $T(u_t,t)=u_t.$ Taking $\frac{\va_1^p}{u_t^{p-1}}$ as the test function and using \cite[Proposition 4.2]{CPBF} we have
 \begin{equation}\label{e627}
     \begin{aligned}
    \int_{\Om}&\frac{(u_t+t)^r\va_1^p}{u_t^{p-1}}
		\\&= \int_{\mathbb R^N\times\mathbb R^N} \frac{|u_t(x)-u_t(y)|^{p-2}(u_t(x)-u_t(y))}{|x-y|^{N+sp}} \left(\frac{\varphi_1(x)^{p}}{u_t(x)^{p-1}} - \frac{\va_1(y)^{p}}{u_t(y)^{p-1}}\right)\nonumber\\
		&\hspace{4.4cm}\leq \int_{\mathbb R^N\times\mathbb R^N} \frac{|\va_1(x)-\va_1(y)|^{p}}{|x-y|^{N+sp}}\nonumber=\la_1 \int_\Om|\varphi_1|^{p}.
     \end{aligned}
 \end{equation}
 % \begin{align}\label{e627}
	% \int_{\Om}\frac{(u_t+t)^r\va_1^p}{u_t^{p-1}}
	% 	&= 	\int_{\mathbb R^N\times\mathbb R^N} \frac{|u_t(x)-u_t(y)|^{p-2}(u_t(x)-u_t(y))}{|x-y|^{N+sp}} \left(\frac{\varphi_1(x)^{p}}{u_t(x)^{p-1}} - \frac{\va_1(y)^{p}}{u_t(y)^{p-1}}\right)\nonumber\\
	% 	&\leq \int_{\mathbb R^N\times\mathbb R^N} \frac{|\va_1(x)-\va_1(y)|^{p}}{|x-y|^{N+sp}}\nonumber\\
	% 	&=\la_1 \int_\Om|\varphi_1|^{p}.
	% \end{align}
where $\va_1$ is the first eigenfunction of $(-\Delta)_p^s.$ Using the strict convexity of the domain and the monotonicity of the solutions $u_t$ as 
given in \cite[Theorem 1.1]{WWLZ} we appeal to \cite[eqn. 8]{DDPLNR} and obtain for any $x\in \Omega_\epsilon$
	\begin{align*}
		\gamma(\inf_{\Omega\setminus \Omega_{\frac{\epsilon}{2}}}\va_1^p)u_t(x)^{r-p+1}&\leq \int_{I_x}u_t^{r-p+1}(\xi)\va_1^p(\xi) d\xi \leq\int_\Omega u_t^{r-p+1}(\xi)\va_1^p(\xi)dx\\
		&\hspace{3cm}\leq	\int_{\Om}\frac{(u_t+t)^r\va_1^p}{u_t^{p-1}} \; \leq \;\la_1 \int_\Om|\varphi_1|^{p}.
	\end{align*}
	where $I_x$ is a measurable set as defined in \cite{DDPLNR}.
	Thus we get that $\|u_t\|_{L^{\infty}(\Omega_\epsilon)}\leq C_\epsilon$ for all $t.$ Now for $t$ large enough, using \eqref{e627} we have the following estimate
	\begin{equation}\label{e628}
		t\int_{\Om_\epsilon}\frac{\va_1^p}{C_\epsilon^{p-1}}\leq t\int_{\Om_\epsilon}\frac{\va_1^p}{u_t^{p-1}}
		\leq t\int_{\Om}\frac{\va_1^p}{u_t^{p-1}}
		\leq \la_1 \int_\Om|\va_1|^{p}.
	\end{equation}
	Thus, from \eqref{e628}, we infer that $T(u,t)=u$ has no solution for $t\geq t_0.$ This proves (b).\\[4mm]
	We will next show that  if $u$ solves $T(u,t)=u$ for $t\in[0,\infty),$ then $\|u\|_{{\infty}}\leq M$ (independent of $t$) and this verifies condition $(a)$. We proceed as in \cite{JSBG}.
	On the contrary let us assume that there exists $t_{k}\in [0,\infty)$ such that for the corresponding solutions $u_{k}$, $\|u_{k}\|_{L^{\infty}(\Omega)}\rightarrow \infty$. Denote $M_{k}=\|u_{k}\|_{L^{\infty}(\Omega)}\rightarrow \infty$ and let $x_{k}\in \Om$ be points with $M_{k}=u_{k}(x_{k})$. First we claim that, up to a subsequence,
	\begin{equation*}
		\frac{t_k}{\|u_{k}\|_{L^{\infty}(\Omega)}}\rightarrow 0\text{ as }k \rightarrow\infty.
	\end{equation*}
	Indeed, without loss of generality, we assume that $t_k>0$ for all $k$ and $t_k\rightarrow\infty$. Define $w_k:=\frac{u_k}{t_k}$ and $\la_k:=t_k^{r-p+1}$. Then, it is easy to check that $(w_k, \lambda_k)$ satisfies weakly 
	\begin{equation*}
	    \begin{aligned}
	      (-\Delta)_p^s w_k = \lambda_k   (w_k+1)^r
	    \end{aligned}
	\end{equation*}
	Then from the comparison principle, we have $w_k\geq \bar w_k$ where 
	\begin{equation}\label{e629}
	\begin{aligned}
	  (-\Delta)_p^s \bar w_k = \lambda_k   \mbox{ in } \Omega\;\;\; \mbox{ and } \;\;\;\; \bar w_k=0 \mbox{ in } \; \Omega^c
	\end{aligned}
	\end{equation}
	Suppose $\sup_k\|\bar w_{k}\|_{L^{\infty}(\Omega)}$ is bounded by $C$, then using the weak formulation of  \eqref{e629}, we get
	\begin{equation}\label{e630}
		\|\bar w_k\|_{W_0^{s,p}(\Om)}^p\leq C\la_k|\Omega|.
	\end{equation}
Let $\phi\geq 0$ be a nontrivial function in $C_c^\infty(\Omega),$ then  using H\"older's inequality we get 
	\begin{align}\label{e631}
	\displaystyle	0<\int_{\Om}\phi&\leq\frac{1}{\la_k}\int_{\mathbb R^N\times\mathbb R^N} \frac{|\bar w_k(x)-\bar w_k(y)|^{p-2}(\bar w_k(x)-\bar w_k(y))(\phi(x) - \phi(y))}{|x-y|^{N+sp}}\nonumber\\
	\displaystyle\leq	& \, {\la_k^{-1}}\| \bar w_k\|_{W^{s,p}_{0}(\Om)}^{p-1}\|\phi\|_{W^{s,p}_{0}(\Om)} \; \leq {C_1 \la_k^{-1/p}}\rightarrow 0,
	\end{align}	
	which is a contradiction. Thus $\sup_k\|\bar w_{k}\|_{L^{\infty}(\Omega)}$ must be unbounded. Therefore we have 
	\[\displaystyle
	\sup_k\| w_{k}\|_{L^{\infty}(\Omega)}\geq \sup_k\|\bar w_{k}\|_{L^{\infty}(\Omega)}=\infty.\]
 This verifies the claim since $\|w_k\|_\infty=\frac{\|u_k\|_\infty}{t_k}.$
	 Now we introduce the Gidas-Spruck translated function 
	\begin{equation*}
		v_{k}(y)=\frac{u_{k}(x_{k}+\mu_{k}y)}{M_{k}},
		\quad y\in \Omega^{k}
	\end{equation*}
	where
	\begin{equation*} 
		\Omega^{k}=\left\lbrace y\in \mathbb{R}^{N}:x_{k}+\mu_{k}y\in \Omega\right\rbrace 
		\text{ and } \mu_{k}=M^{-\frac{r-(p-1)}{sp}}_{k}.
	\end{equation*}
	A straightforward calculation yields
	\begin{equation}\label{e632}
		(-\Delta)^s_p v_{k}(x)= \left(v_{k}(x)+\frac{t_{k}}{M_{k}}\right)^r,
	\end{equation}
 Since $v_k$'s are also viscosity solutions of (\ref{e632}) and as $\thickspace \frac{t_{k}}{M_{k}}\rightarrow 0,$	using the arguments as in the proof of Theorem \ref{t64}, we can pass through the limit to get a contradiction to the nonexistence results for the sub-critical problem \eqref{p62}.  Hence, if $T(u,t)=u$ then $\|u\|_{L^{\infty}(\Omega)}$ is bounded independent of $t.$ Finally, we use Proposition \ref{pr61} with the proper choice of $R_1$ and $R_2$ and conclude that $deg(S_0, B_{R_2}\setminus\overline{B}_{R_1},0)=-1.$
\end{proof}
\noi Next we prove the main result of our paper for which we proceed as in \cite{MCPDRS}. We determine the degree of $S(\gamma,.)$ by connecting $S(\gamma,.)$ and $S(0,.)$ using the homotopy invariance of degree with respect to $\gamma$. Since $deg(S_0, B_{R_2}\setminus\overline{B}_{R_1},0)\neq 0,$ in particular, this will imply that $S(\gamma,w)$ has a solution $w$ satisfying $ R_{1}<\|w\|_{L^{\infty}(\Om)}<R_{2}$. Finally, we will show that the solution hence obtained for \eqref{p66} is infact positive for $\gamma$ small enough.
\begin{theorem}\label{t65}
For $p\ge 2$, the problem \eqref{p66} admits a positive solution for $\gamma\in[0,\gamma_0].$ 
\end{theorem}
\begin{proof}
	We prove the theorem in two steps:\\
	\textbf{STEP I: } There exists a $\gamma_0>0$ such that $deg(S(\gamma,\cdot), B_{R_2}\setminus \overline{B}_{R_1},0)=-1$ for all $\gamma \in [0,\gamma_{0}]$.
	
	If  $S(\gamma,w)\neq 0$ for $\|w\|_\infty \in \{R_1, R_2\},$ then using Lemma \ref{l67} and the homotopy invariance of degree we know that  $deg(S(\gamma,\cdot), B_{R_2}\setminus\overline{B}_{R_1},0)=-1.$ Suppose that $S(\gamma_n,w_n)=0$ for some sequence  $\gamma_n \rightarrow 0$ and $\|w_n\|\in \{R_1,R_2\}.$ Since $K \tilde{F}(\gamma, \cdot) : C(\overline{\Om})\rightarrow  C(\overline{\Om}) $ is compact, we can find a function $w_0\in  C(\overline{\Om})$ with $\|w_0\|\in \{R_1,R_2\}$ and $S(0,w_0)=0$. This contradicts the previous lemma \ref{l67} and hence Step I is proved. \\
	Clearly, Step I implies that the set of solutions of $S(\gamma,w)=0$ is nonempty for $\gamma$ small enough. If $w$ is a positive function solving the equation $S(\gamma,w)=0,$ then $w$ solves the PDE \eqref{p66}. With this observation, the proof of Theorem \ref{t65} is complete if we prove Step II  given below. \\
	\textbf{ STEP II: } For $\gamma$ small enough, $S(\gamma, w)=0$ for some $w$ in $B_{R_2}\setminus \overline{B}_{R_1}$ implies $w>0.$\\
	  Let $S(\gamma_n, w_n)=0$ for $\gamma_n \rightarrow 0$ and $w_n\in B_{R_2}\setminus \overline{B}_{R_1}$. Since $\|w_n\|_{C(\overline{\Om})}$ is bounded, using \cite[Theorem 1.1 ]{t24} and \cite[Theorem 1.1 ]{IMSS1}, we have 
	$$\|w_n\|_{C^\alpha(\overline{\Om})}\leq C \mbox{ and }\left \|\frac{w_n}{d^s(x)}\right\|_{C^\alpha(\overline{\Omega})}\leq C.$$
	Thus, by Ascoli Arzela theorem, up to a subsequence 
	$$  w_n \rightarrow w_0 \mbox{ in } C^0(\overline{\Omega}) \;\; \mbox{  and    }\;\; \frac{w_n}{d^s(x)}  \rightarrow \frac{w_0}{d^s(x)}  \mbox{ in } C^0(\overline{\Omega}) .$$ 
  We also know that $w_0$ is a non-zero solution of $S(w_0,0)=0.$ By regularity results $w_0\in C_{d^s}^0(\overline{\Omega})$ and Hopf type lemma \cite{HDLQ}, $\inf_{x\in \Om} \frac{w_0(x)}{d^s(x)}>0.$ 
  %Let $C_+=\{u\in C_{d}^0(\Omega) : u(x)\geq 0\},$ then $w_0$ is in the interior of $C_+.$
	Finally, the positivity of $w_n$ follows by using the above uniform convergence and the fact that $\displaystyle\frac{w_0(x)}{d^s(x)} >0$ in $\overline{\Om}.$ Hence we conclude that \eqref{p66} admits a positive solution for $\gamma\in[0,\gamma_0].$ 
	\end{proof}

{\noi \textbf{Proof of the Theorem \ref{t61}:} As per the discussion presented at the beginning of section \ref{deg}, we can infer that the existence of positive solutions for $(P^{\mu})$ is equivalent to the existence of positive solutions for \eqref{p66}. Since theorem \ref{t65} has been established, it follows that problem $(P^{\mu})$ has a positive solution for all $\mu\in (0,\mu_0)$.}

\begin{remark}
 	We remark that with slight modification, the above theorem can be proved for more general Dirichlet problems like $\fpl u=\la f(u)$ in $\Omega$ and $u=0$ in $\RR^N\setminus \Omega$ where $f: [0,\infty)\rightarrow \RR$ is such that $f(0)<0,$ $f(s)\geq 0$ for $s>>1$ and an additional growth assumption $\lim_{s\rightarrow \infty} \frac{f(s)}{s^{q-1}}=b$ for some $b>0$ and $q\in (p-1,\pst-1).$ The details are exactly as discussed in \cite{MCPDRS}.
\end{remark}
\begin{remark}
Notably, in the local case of the $p$-Laplacian when $p\geq 2$, Theorem \ref{t61} represents a unique contribution as it employs degree theory, which has not been utilized previously to establish similar results. Prior research by \cite{MCPDRS} could only examine the case where $p\in (1,2]$ due to the absence of a uniform $L^\infty$ bound. However, by employing the ideas discussed previously for the local scenario, we can now expand the results of \cite{MCPDRS} to encompass the case of p Laplacian when $p>2.$
\end{remark}
% Additionally, we would like to emphasize that Theorem \ref{t61}, which we prove in this work, is a novel result even in the local case of the $p$-Laplacian when $p\geq 2$.
% Theorem \ref{t61} introduces novel findings in the local case of the $p$-Laplacian when $p\geq 2$, and it distinguishes itself by employing degree theory.
% By the discussion at the begining of the section  \ref{deg} we can conclude that existence of positive solutions for $(P^{\mu})$ is equivalent to the existence of positive solution for \eqref{p66}. Since we have proved theorem \ref{t65},hence, problem $(P^{\mu})$ admits a positive solution for all $\mu\in (0,\mu_0).$  
 \section{ Semipositone Fractional p-Laplace Problem with Critical Growth }
Semipositone problems with critical exponents have been extensively studied in the literature. It is well-established that, by using the Pohozaev identity for the p-Laplacian, one can prove the non-existence of positive solutions for the equation $-\Delta_p u = \mu(u^{p^*-1}-1)$ subject to zero Dirichlet boundary conditions in any star-shaped domain $\Omega$. For the fractional Laplace equation, where the Pohozaev identity is available (as shown in \cite{RS12}), similar non-existence results can be proven for the semipositone critical exponent problem. However, for the fractional p Laplacian, a Pohozaev identity in a bounded domain has not been established, which makes it difficult to confirm this fact completely. Nonetheless, recent results in \cite{amb23} have made progress towards understanding the Pohozaev identity for the fractional p Laplacian in $\mathbb{R}^N$.

Despite this uncertainty, we can still draw motivation from the work of \cite{PSS} and consider a perturbed multiparameter problem in the spirit of the Brezis-Nirenberg problem as described in $(\ref{nscpt1})$. In \cite{PSS}, for the local case, a positive solution to this perturbed problem was shown to exist. In this section, we extend this study to a similar problem involving the fractional p Laplacian operator.

% the existence of positive solution for semipositone problem with a critical exponent term. 
 
%  \textcolor{red}{Uttam:  please check the validity of the lines.\textcolor{blue}{If we have a result of the type Theorem 1.1 of \cite{ROSOTON}, for the fractional p Laplacian we can prove  that $(P^\mu)$ admits no solution. But the research for Pohozaev idenity for fractional p Laplacian is still in its intial stage and a recent preprint shows that it is true in $\RR^N.$(\cite{NEWRESULT}) The recent Pohozaev identity for fractional p-Laplacian is on the whole of $\RR^n$ and hence still not easy to apply for showing the nonexistence of positive solution. So we will go ahead with NA assumptions. The lines below are fine.}
% Next we focus our attention to the critical exponent problem,  $r=\pst-1$, the problem $(P_\la^\mu)$ is not expected to admit a  positive solution. This result can easily be verified for the p-Laplacian as an application of Pohazaev identity.  But the result could not be proved for the fractional p Laplacian due to the lack of Pohazaev type identity for $(-\Delta)_p^s.$  }We hence consider a  perturbed critical exponent problem and prove the existence of a positive solution for a certain range of $\lambda$ and $\mu$ for the same. So we study the Brezis-Nirenberg type critical semipositone fractional Laplacian problem
 
The main objective of this section is to establish the existence of a positive solution to the $p$-superlinear semipositone problem with critical growth involving the fractional $p$-Laplace operator, as presented in equation ($P_\la^\mu$) for the conditions $p\ge2$, $N\ge sp^{2}$, and $\lambda\in (0,\lambda_{1})$. Therefore, from now on in this section, we will assume that these conditions hold.  To prove Theorem \ref{nscpt1}, we follow a similar approach to that of \cite{PSS} by utilizing variational tools as the background framework for the mountain pass lemma. However, a significant challenge we face in this study is demonstrating the positivity of the solution. This issue is resolved through the implementation of Hopf's Lemma, which is discussed in \cite{IASLKM}, specifically when $p\geq 2$.
 We now  consider the modified problem for $\lambda,\mu>0$
\begin{eqnarray}\label{nscp6}  \left.
	\begin{array}{rl}
		(-\Delta)_{p}^s u &=\lambda u^{p-1}_{+}+u_{+}^{p^{*}_{s}-1}-\mu f(u) \text{ in }\Omega, \\	
		u &=0 \hspace{3.85cm} \text{ on }\Omega^{c},
	\end{array} 
	\right\}
\end{eqnarray} 
where $u_{+}(x)=\max\{ u(x),0\}$ and
\[ f(t)=
\begin{cases}
	1 &\quad t\ge 0,\\
	1-|t|^{p-1} &\quad -1<t<0, \\
	0 &\quad t\le -1.
\end{cases}
\]
We are interested in finding the critical points of the $C^{1}$-functional 
\begin{equation*}
	\begin{split}
		E_{\mu}(u)=	 &\frac{1}{p} \int_{\mathbb{R}^N \times \mathbb{R}^N} \dfrac{|u(x)-u(y)|^{p}}{|x-y|^{N+sp}}\,dx\,dy +\int_{\Omega}\left( -\frac{\lambda u^{p}_{+}}{p}-\frac{u^{p^{*}_{s}}_{+}}{p^{*}_{s}}\right)\,dx 
		\\
		&+\mu \left[ \int_{\{-1<u<0\}}\left( u-\frac{u|u|^{p-1}}{p}\right) \,dx-\left(1- \frac{1}{p}\right) |\{u\leq -1\} |\right]
		\\
		&\hspace{8cm}+\mu \int_{\{u\ge 0\}} u\,dx
	\end{split}
\end{equation*}
where $|.|$ denotes the Lebesgue measure in $\mathbb{R}^{N}$. For all $v\in 	D_{0}^{s,p}(\Omega)$, we also have 
\begin{equation*}
    \begin{aligned}
          \langle DE_{\mu}(u){,}v\rangle= &\int_{\mathbb{R}^N \times \mathbb{R}^N} \dfrac{|u(x)-u(y)|^{p-2}(u(x)-u(y))(v(x)-v(y))}{|x-y|^{N+sp}}\,dx\,dy
          \\&\hspace{1.8cm}
			+\mu\left[ \int_{\{u\ge 0\}} v\,dx+ \int_{\{-1<u<0\}}\left(1- |u|^{p-1}\right)v \,dx \right] \quad 
          \\
		&\hspace{5.8cm}+\int_{\Omega}\left(\lambda u_{+}^{p-1}v-u_{+}^{p^{*}_{s}-1}v\right)\,dx .			
    \end{aligned}
\end{equation*}We now define $S_{s,p}$, the best constant in Sobolev inequality by 
\begin{equation}\label{nscp7}
	S_{s,p}=\inf_{u\in D^{s,p}_{0}(\Omega)\setminus{\{0\}}}	\dfrac{\displaystyle\int_{\mathbb{R}^N \times \mathbb{R}^N} \dfrac{|u(x)-u(y)|^{p}}{|x-y|^{N+sp}}\,dx\, dy}{\left( \displaystyle\int_{\Omega}|u(x)|^{p^{*}_{s}}\,dx\right)^{p/p^{*}_{s}}}	
\end{equation}
We encourage the readers to go through \cite{BNMPY} for more results regarding the minimization problem \eqref{nscp7}.
We also recall the definition of the fractional gradient here.

\begin{definition}\cite{BSNCC}\label{nsd1} $(s,p)$ gradient of a function $v\in D_{0}^{s,p}(\mathbb{R}^{N})$ is defined as
	\begin{equation*}
		|D^{s}v(x)|^{p}=\int_{\mathbb{R}^{N}} \frac{|v(x+h)-v(x)|^{p}}{|h|^{N+sp}}\,dh.
	\end{equation*}   
\end{definition}
\noi We note that $(s,p)$ gradient is well defined a.e in $\mathbb{R^{N}}$ and $|D^{s}v| \in L^{p}(\mathbb{R^{N}})$. Next we recall the concentration compactness theorem \cite{BSNCC}.
\begin{theorem}\label{cct}\cite[Theorem 1.1]{BSNCC}
Let $(u_{n})\subset D^{s,p}_{0}(\mathbb{R^{N}})$ be a weakly convergent subsequence with weak limit $u$. Then there exist two bounded measures $\kappa$ and $\nu$, an atmost enumerable set of indices $I$, and positive real numbers $\kappa_{i},\nu_{i}$, points $x_{i}\in \overline{\Omega},i\in I$, such that the following convergence hold weakly$^{*}$ in the sense of measures.
\begin{align*}		
|D^{s}u_{n}|^{p}dx &\rightharpoonup \kappa\ge |D^{s}u|^{p}dx+\sum_{i\in I}\kappa_{i}\delta_{x_{i}}\\
|u_{n}|^{p^{*}_{s}}dx &\rightharpoonup \nu=|u|^{p^{*}_{s}}dx+\sum_{i\in I}\nu_{i}\delta_{x_{i}}\\
S^{1/p}_{s,p}\nu^{1/p^{*}_{s}}_{i}&\leq \kappa^{1/p}_{i}\quad \text{for all}\quad i\in I
\end{align*}
where $S_{s,p}$ as in \eqref{nscp7}.
\end{theorem}
\noi We start this section by finding the level of PS-condition.
\begin{lemma}\label{nscpl1}	
For any fixed $\lambda,\mu >0$, $ E_{\mu}$ satisfies the $(PS)_{c}$ condition for all
	\begin{equation}\label{nscp8}
		c<\frac{s}{N}S^{N/sp}_{s,p}-\left( 1-\frac{1}{p}\right) \mu|\Omega|.
	\end{equation}	
\end{lemma}
\begin{proof}
Let $(u_{n})$ be a $(PS)_{c}$ sequence. We can show that sequence  $(u_{n})$ is bounded in $D_{0}^{s,p}(\Omega)$ following similar steps of \cite[Lemma 2.1]{PSS}.
Since $(u_{n})$ is bounded so is $(u_{n+})$, a renamed subsequence which converges to some $u_{+}\ge 0$ weakly in $D_{0}^{s,p}(\Omega)$, strongly in $L^{q}(\Omega)$ for all $q\in [1,p^{*}_{s})$, $a.e.$ in $\Omega$ and  
	\begin{equation}\label{nscp14}
		|D^{s}u_{n+}|^{p}dx \rightharpoonup \kappa, \quad u_{n+}^{p^{*}_{s}}dx \rightharpoonup \nu
	\end{equation}
	The convergence holds weakly* in the sense of measure, where $\kappa$ and $\nu$ are bounded measures. Using concentration compactness theorem \ref{cct} there exists a countable index set $I$ and  positive real numbers $\kappa_{i},\nu_{i}$, points $x_{i}\in \overline{\Omega},\thinspace  i\in I$ such that 
	\begin{equation}\label{nscp15}
		\kappa\ge |D^{s}u_{+}|^{p}dx+\sum_{i\in I}\kappa_{i}\delta_{x_{i}},\quad \nu=u_{+}^{p^{*}_{s}}dx+\sum_{i\in I}\nu_{i}\delta_{x_{i}}	
	\end{equation}
	and $S^{1/p}_{s,p}\nu^{1/p^{*}_{s}}_{i}\leq \kappa^{1/p}_{i}$ for all $i\in I $. Our aim is to prove $I=\emptyset$. Suppose by contradiction we fix a concentration point $x_{i}$. We define a smooth function $\varphi:\mathbb{R^{N}}\rightarrow [0,1] $ such that $\varphi(x)	=1$ for $|x|\leq 1$ and $\varphi(x)=0$ for $|x|\ge 2$ and for $i\in I$ and $\rho >0$ set
	\begin{equation*}
		\varphi_{i,\rho}(x)=\varphi\left(\frac{x-x_{i}}{\rho}\right), \thickspace x\in \mathbb{R}^{N}.
	\end{equation*}
	Clearly $\varphi_{i,\rho}:\mathbb{R}^{N}\rightarrow [0,1] $ is a smooth function. We note that the sequence $(\varphi_{i,\rho}u_{n+})$ is bounded in $D^{s,p}_{0}(\Omega)$. Taking $v=\varphi_{i,\rho}u_{n+}$ in $	\langle DE_{\mu}(u_{n}){,}v\rangle$, we get
	\begin{equation}\label{nscp16}
		\begin{split}
			&\int_{\mathbb{R}^N \times \mathbb{R}^N} \dfrac{|u_{n}(x)-u_{n}(y)|^{p-2}(u_{n}(x)-u_{n}(y))(\varphi_{i,\rho}u_{n+}(x)-\varphi_{i,\rho}u_{n+}(y))}{|x-y|^{N+sp}}\,dx\,dy\\
			&=\int_{\Omega}\left(\lambda u_{n+}^{p-1}\varphi_{i,\rho}u_{n+}+u_{n+}^{p^{*}_{s}-1}\varphi_{i,\rho}u_{n+}\right)\,dx 
			-\mu \int_{\Omega} \varphi_{i,\rho}u_{n+}\,dx + o_{n}(1)\|u_{n+}\|.	 
		\end{split}
	\end{equation}
	 The term on the left-hand side of \eqref{nscp16} can be estimated as
	\begin{equation}\label{nscp17}
		\begin{split}
			&\int_{\mathbb{R}^N \times \mathbb{R}^N} \dfrac{|u_{n}(x)-u_{n}(y)|^{p-2}(u_{n}(x)-u_{n}(y))(\varphi_{i,\rho}u_{n+}(x)-\varphi_{i,\rho}u_{n+}(y))}{|x-y|^{N+sp}}\,dx\,dy\\
			&\ge \int_{\mathbb{R}^N \times \mathbb{R}^N} \dfrac{|u_{n}(x)-u_{n}(y)|^{p-2}(u_{n}(x)-u_{n}(y))(\varphi_{i,\rho}(x)-\varphi_{i,\rho}(y))}{|x-y|^{N+sp}}u_{n+}(y) \,dx\,dy  \\
			&\hspace{6cm}+\int_{\mathbb{R}^N \times \mathbb{R}^N}\dfrac{|u_{n+}(x)-u_{n+}(y)|^{p}}{|x-y|^{N+sp}} \varphi_{i,\rho}(x)\,dx\,dy.
		\end{split}
	\end{equation}	
\noi For the second term on right hand side of \eqref{nscp17}, for a fixed $x$, taking the transformation
$y-x=h$ and using the definition \eqref{nsd1} and \eqref{nscp14}, we have 
	\begin{equation}\label{nscp18}
		\int_{\mathbb{R}^N \times \mathbb{R}^N}\dfrac{|u_{n+}(x)-u_{n+}(y)|^{p}}{|x-y|^{N+sp}} \varphi_{i,\rho}(x)\,dx\,dy \rightarrow \int_{\mathbb{R}^{N}} \varphi_{i,\rho} \,d\kappa.	
	\end{equation}
	For the first term in right hand side of \eqref{nscp17}, using the definition of $(s,p)$ gradient, we have 
	\begin{equation}\label{nscp19}
		\begin{split}
			&\int_{\mathbb{R}^N \times \mathbb{R}^N} \dfrac{|u_{n}(x)-u_{n}(y)|^{p-2}(u_{n}(x)-u_{n}(y))(\varphi_{i,\rho}(x)-\varphi_{i,\rho}(y))}{|x-y|^{N+sp}}u_{n+}(y) \,dx\,dy\\
			&\hspace{7cm}\leq C \left(\int_{\mathbb{R}^{N}} |u_{n+}|^{p} |D^{s}\varphi_{i,\rho}|^{p\,dx}\right)^{\frac{1}{p}}.		
		\end{split} 		
	\end{equation}
	Now we consider the first term on right hand side of \eqref{nscp16} i.e.,
	\begin{equation*}
		\int_{\Omega}\left(\lambda u_{n+}^{p}\varphi_{i,\rho}+u_{n+}^{p^{*}_{s}}\varphi_{i,\rho}-\mu\varphi_{i,\rho} u_{n+}\right)\,dx . 
	\end{equation*}
	Using \eqref{nscp14} we have
	\begin{equation}\label{nscp20}
		\int_{\Omega} \varphi_{i,\rho}u_{n+}^{p^{*}_{s}}\,dx \rightarrow \int_{\Omega} \varphi_{i,\rho}\,d\nu.
	\end{equation}
	Now we observe that 
	\begin{equation*} 
		\begin{split}
			\bigg|\int_{\mathbb{R}^N \times \mathbb{R}^N}& \dfrac{|u_{n}(x)-u_{n}(y)|^{p-2}(u_{n}(x)-u_{n}(y))(\varphi_{i,\rho}(x)-\varphi_{i,\rho}(y))}{|x-y|^{N+sp}}u_{n+}(y) \,dx\,dy \\
			&\hspace{5.8cm}+\int_{\Omega}\left(-\lambda u_{n+}^{p}\varphi_{i,\rho}+\mu\varphi_{i,\rho} u_{n+}\right) \,dx \bigg|  \end{split}
	\end{equation*}
	\begin{equation*}
		\le C \left[  \left(\int_{\mathbb{R}^{N}} |u_{n+}|^{p} |D^{s}\varphi_{i,\rho}|^{p}\,dx\right)^{\frac{1}{p}} + \int_{\Omega\cap B_{2\rho}(x_{i})} u_{n+}^{p}\,dx + \mu \int_{\Omega\cap B_{2\rho}(x_{i})} u_{n+}\,dx\right] .
	\end{equation*}
	We also have 
	\begin{equation}\label{nscp21}
		\int_{\Omega\cap B_{2\rho}(x_{i})} u_{n+}^{p}\,dx\rightarrow \int_{\Omega\cap B_{2\rho}(x_{i})} u_{+}^{p}\,dx. 
	\end{equation} 
	Next passing to the limit in \eqref{nscp16} and using \eqref{nscp17}-\eqref{nscp21}, we get
	\begin{equation*}
		\begin{split}
			&\int_{\mathbb{R}^{N}} \varphi_{i,\rho} \,d\kappa-\int_{\Omega} \varphi_{i,\rho}\,d\nu\\
			&\leq C\left[  \left(\int_{\mathbb{R}^{N}} |u_{n+}|^{p} |D^{s}\varphi_{i,\rho}|^{p}\,dx\right)^{\frac{1}{p}} + \int_{\Omega\cap B_{2\rho}(x_{i})} u_{+}^{p}\,dx + \mu  \int_{\Omega\cap B_{2\rho}(x_{i})} u_{+}\,dx \right].	
		\end{split}
	\end{equation*}
	Now letting $\rho\rightarrow 0$, the right hand side of the above inequality goes to $0$. This implies $\kappa_{i}\leq \nu_{i}$, which together with $\nu_{i}>0$ and $S^{1/p}_{s,p}\nu^{1/p^{*}_{s}}_{i}\leq \kappa^{1/p}_{i}$ gives $\nu_{i}\ge S^{N/sp}_{s,p}$. On the other hand, similar to  \cite[Lemma 2.1]{PSS}  and using \eqref{nscp14} and \eqref{nscp15} we get
	\begin{equation*}
		\nu_{i}	\leq \frac{N}{s}\left[ \left(1- \frac{1}{p}\right) \mu|\Omega|+c\right] <  S^{N/sp}_{s,p}
	\end{equation*}
	a contradiction. Hence $I=\emptyset$ and 
	\begin{equation}\label{nscp22}
		\int_{\Omega} u_{n+}^{p^{*}_{s}}\,dx \rightarrow \int_{\Omega} u_{+}^{p^{*}_{s}}\,dx.	
	\end{equation}
	Next passing to further subsequence, $u_{n}$ converges weakly to $u$ i.e., $u_{n}\rightharpoonup u$ in $D^{s,p}_{0}(\Omega)$, strongly in $L^{r}(\Omega)$ for $r\in [1, p^{*}_{s})$ and a.e. in $\Omega$. \\
	We note that $u$ weakly solves the problem \eqref{nscp6} and using the Brezis-Lieb type lemma and \eqref{nscp22} we can show that $u_{n}\rightarrow u$ in $ D^{s,p}_{0}(\Omega)$. This completes the proof of Lemma \ref{nscpl1}.
\end{proof}
\noi
Next, we state some results regarding the minimization problem \eqref{nscp7} , which is employed to establish mountain pass results, and introduce $u_{\epsilon,\delta}$ in the same way as depicted in   \cite[Equation (2.16)]{BNMPY}. For the sake of completeness, we shall discuss the definition here.
% Next, we state some results regarding the minimization problem \eqref{nscp7} (See \cite{BNMPY}) that is used to prove mountain pass results. For any $\epsilon>0$, the function
% \begin{equation}
% 	U_{\epsilon}(x)=\frac{1}{\epsilon^{(N-sp)/p}}U\left( \frac{|x|}{\epsilon}\right) 
% \end{equation}
% is a minimizer for $S_{s,p}$ and also satisfies the following equations
% \begin{equation*}
% 	(-\Delta)^{s}_{p} U=U^{p^{*}_{s}-1},\thickspace \text{and}\thickspace \|U\|^{p}=\|U\|_{L_{p^{*}_{s}}(\mathbb{R}^{N})}^{p^{*}_{s}}=S^{N/sp}_{s,p}.
% \end{equation*}
% \textcolor{teal}{$U$ is not defined here. It is coming from \cite[prop 2.1]{BNMPY}. " 
Let $U$ be a normalised radially symmetric nonnegative decreasing minimizer  of $S_{s,p}$ ( see \cite[propsition 2.1]{BNMPY}) which satisfies 
\begin{equation}\label{min}
	(-\Delta)^{s}_{p} U=U^{p^{*}_{s}-1},\thickspace \text{and}\thickspace \|U\|^{p}=\|U\|_{L_{p^{*}_{s}}(\mathbb{R}^{N})}^{p^{*}_{s}}=S^{N/sp}_{s,p}.
\end{equation} 
For any $\epsilon>0$, 
\begin{equation}
	U_{\epsilon}(x)=\frac{1}{\epsilon^{(N-sp)/p}}U\left( \frac{|x|}{\epsilon}\right) 
\end{equation}
denotes the associated family of minimisers for $S_{s,p}$ and also satisfies \eqref{min}.
% By \cite[propsition 2.1]{BNMPY}, we fix a radially symmetric nonnegative decreasing minimiser $U$ for $S_{s,p}$. Multiplying by a positive constant if necessary we have the following \begin{equation}\label{min}
% 	(-\Delta)^{s}_{p} U=U^{p^{*}_{s}-1},\thickspace \text{and}\thickspace \|U\|^{p}=\|U\|_{L_{p^{*}_{s}}(\mathbb{R}^{N})}^{p^{*}_{s}}=S^{N/sp}_{s,p}.
% \end{equation} 
% For any $\epsilon>0$, the function
% \begin{equation}
% 	U_{\epsilon}(x)=\frac{1}{\epsilon^{(N-sp)/p}}U\left( \frac{|x|}{\epsilon}\right) 
% \end{equation}
% is a minimizer for $S_{s,p}$ and also satisfies \eqref{min}.
Due to the absence of explicit formula for $U$, we have the following asymptotic estimates.
\begin{lemma}\label{7l1}
	There exists $c_{1},c_{2}>0$ and $\theta>1$ such that for all $r\ge 1$,
	\begin{equation*}
		\dfrac{c_{1}}{r^{(N-sp)/(p-1)}}\leq U(r)\leq \dfrac{c_{2}}{r^{(N-sp)/(p-1)}}
	\end{equation*}
	and 
	\begin{equation*}
		\frac{U(\theta r)}{U(r)}\leq \frac{c_{2}}{c_{1}}\dfrac{1}{\theta^{(N-sp)/(p-1)}}.
	\end{equation*}
\end{lemma}
\noi We have some auxiliary estimates from \cite{BNMPY}. For $\epsilon,\delta>0$, and $\theta$ as in Lemma \ref{7l1}, set
\begin{equation*}
	m_{\epsilon,\delta}=\dfrac{U_{\epsilon}(\delta)}{U_{\epsilon}(\delta)-U_{\epsilon}(\theta\delta)}
\end{equation*}
and 
\begin{equation*} 
	g_{\epsilon,\delta}(t) =
	\begin{cases}
		0      & \quad \text{if}\thickspace 0\leq t\leq U_{\epsilon}(\theta\delta)\\
		m^{p}_{\epsilon,\delta}(t-U_{\epsilon}(\theta\delta))& \quad \text{if}\thickspace U_{\epsilon}(\theta\delta)\leq t\leq U_{\epsilon}(\delta) \\
		t+U_{\epsilon}(\delta)(m^{p-1}_{\epsilon,\delta}-1) & \quad \text{if}\thickspace t\ge U_{\epsilon}(\delta),
	\end{cases}
\end{equation*}
and let 
\begin{equation}
	G_{\epsilon,\delta}(t) =\int_{0}^{t}g^{\prime}_{\epsilon,\delta}(\tau)^{\frac{1}{p}}\,d\tau=
	\begin{cases}
		0      & \quad \text{if}\thickspace 0\leq t\leq U_{\epsilon}(\theta\delta)\\
		m_{\epsilon,\delta}(t-U_{\epsilon}(\theta\delta))& \quad \text{if}\thickspace U_{\epsilon}(\theta\delta)\leq t\leq U_{\epsilon}(\delta) \\
		t & \quad \text{if}\thickspace t\ge U_{\epsilon}(\delta).
	\end{cases}
\end{equation} 
The functions $g_{\epsilon,\delta}$ and $G_{\epsilon,\delta}$ are nondecreasing and absolutely continuous. Consider the radially symmetric nonincreasing function
\begin{equation*}
	u_{\epsilon,\delta}(r)=G_{\epsilon,\delta}(U_{\epsilon}(r))
	,
\end{equation*}
which satisfies
\begin{equation*} u_{\epsilon,\delta}(r)=
	\begin{cases}
		U_{\epsilon}(r) &\quad \text{if}\thickspace r\leq \delta\\
		0 &\quad \text{if}\thickspace r\ge \theta\delta.
	\end{cases}
\end{equation*}
 Now similar to \cite[Lemma 3.1]{PSS} and thanks to $N\geq sp^2$, we can show that for all sufficiently small $\mu >0,$ $E_{\mu}$ has a uniformly  positive mountain pass level below the threshold for compactness given in Lemma \ref{nscpl1} . \begin{lemma}\label{nscpl2} There exist positive $\mu_{0},\rho,c_{0}>0,R>\rho$ and $\beta < \frac{s}{N}S^{N/sp}_{s,p}$ such that the following hold for all $\mu \in (0,\mu_{0})$ and $\la\in (0,\la_1)$:
	\begin{itemize}
 		\item[${(i)}$] $E_{\mu}(0)=0$ and $E_{\mu}(u)\ge c_{0}$ for all $u$ such that $ \|u\|=\rho$,
		\item[${(ii)}$] $E_{\mu}(tu_{\epsilon,\delta})\leq 0$ for all $t\ge R$ and $\epsilon\leq \delta /2$ and $\delta \in (0,1]$,
		\item[${(iii)}$] denoting by $\Gamma=\left\lbrace \gamma\in C([0,1],D^{s,p}_{0}(\Omega)):\gamma(0)=0,\gamma(1)=Ru_{\epsilon,\delta}\right\rbrace $ the class of paths joining the origin to $Ru_{\epsilon,\delta}$,
		\begin{equation}\label{nscp23}
			c_{0}\leq c_{\mu}:=\inf_{\gamma \in \Gamma}	\max_{0\leq t\leq 1} E_{\mu}(\gamma(t))\leq \beta-\left( 1-\frac{1}{p}\right) \mu|\Omega|
		\end{equation}
		for all sufficiently small $\epsilon >0$,
		\item[${(iv)}$] $E_{\mu}$ has a critical point $u_{\mu}$ at the level $c_{\mu}$.
	\end{itemize}
\end{lemma}
From the above lemma we conclude that there exists $u_\mu$ which is a weak solution of \eqref{nscp6}. Now we shall prove some more properties of $u_\mu.$
\begin{lemma}\label{nscpl3}
There exists $\mu_{*}\in (0,\mu_{0}]$  such that the following hold for all $\mu\in (0,\mu_{*})$:
	\begin{itemize}
		\item[${(i)}$] $u_{\mu}$ is uniformly bounded in $D^{s,p}_{0}(\Omega)$,
		\item[${(ii)}$] $ \int_{E} |u_{\mu}|^{p^{*}_{s}} \,dx \rightarrow 0 \thickspace \text{as} \thickspace |E|\rightarrow 0,\thickspace \text{uniformly in} \thickspace \mu$,
		\item[${(iii)}$] $u_{\mu}$ is uniformly bounded in $C_{d}^{0,\alpha}(\overline{\Omega})$ for some $\alpha \in (0,s]$.	
	\end{itemize}	
\end{lemma}

\begin{proof}
As mentioned in the proof of Lemma \ref{nscpl1} the sequence $\{u_\mu\}$ is uniformly bounded in $D_0^{s,p}(\Omega)$ for $\mu\in (0,\mu_0)$ which proves $(i).$\\
Next we shall prove $(ii).$ On contrary let us suppose that for some sequence ${\mu_j}\rightarrow 0$ and $|E_j|\rightarrow 0$ we have $\liminf _{j\rightarrow \infty} \int_{E_j}|u_{\mu_j}|^{p_s^*} >0.$ Since $(u_{\mu_{j}})$ is bounded,  a renamed subsequence $(u_{\mu_{j+}})$  converges to some $u_{+}\ge 0$ weakly in $D_{0}^{s,p}(\Omega)$. Following the argument as in Lemma \ref{nscpl1} we get 
	\begin{equation}\label{nscp28}
		\int_{\Omega} u_{\mu_{j+}}^{p^{*}_{s}}\,dx \rightarrow \int_{\Omega} u_{+}^{p^{*}_{s}}\,dx.
	\end{equation}
This implies,  $(u_{\mu_{j}})$ then converges to $u$ in $D_{0}^{s,p}(\Omega)$, and so also in $L^{p^{*}_{s}}(\Omega)$. Then
	\begin{equation*}
		\|u_{\mu_{j}}\|_{L^{p^{*}_{s}}({E_{j}})}\leq \| u_{\mu_{j}}-u\|_{L^{p^{*}_{s}}({\Omega})}+\|u\|_{L^{p^{*}_{s}}({E_{j}})}\rightarrow 0,
	\end{equation*}
	which is a contradiction and hence (ii).
 
 Now Theorem \ref{anscpt3} is applicable and $\|u_{\mu_j}\|_\infty$ is uniformly bounded. Since $p\geq 2,$ \cite[Theorem 1.1]{IMSS1} is applicable and $u_{\mu}$ is uniformly bounded in $C_{d}^{0,\alpha}(\overline{\Omega}).$
\end{proof}	
\noi \textbf{Proof of the Theorem \ref{nscpt1}:}
We claim that $u_{\mu}$ is positive in $\Omega$ and hence a solution of \eqref{nscp3}. It is sufficient to show that for every sequence $\mu_{j}\rightarrow 0$, a subsequence $u_{\mu_{j}}$ is positive in $\Omega$. By Lemma \ref{nscpl2}, we have 
\begin{equation*}
    \begin{aligned}
          	E_{\mu_{j}}(u_{\mu_{j}})=	
	c_{\mu_{j}}\ge c_{0}
    \end{aligned}
\end{equation*}
and 
\begin{equation*}
    \begin{aligned}
          	\langle DE_{\mu_{j}}(u_{\mu_{j}}){,}v\rangle =0 \quad \forall
		\thinspace v\in 	D_{0}^{s,p}(\Omega).
    \end{aligned}
\end{equation*}
From Lemma \ref{nscpl3}, we have upto a subsequence $u_{\mu_j}\rightharpoonup u$ in $D^{s,p}_{0}(\Omega).$ Now with the help of the ideas of \cite[Theorem 2.15]{WSS1} and Brezis-Lieb type lemma, we can conclude for $\mu_{j}\rightarrow 0$
\begin{equation*}
    \begin{aligned}
          \int_{\mathbb{R}^N \times \mathbb{R}^N} \dfrac{|u_{\mu_{j}}(x)-u_{\mu_{j}}(y)|^{p}}{|x-y|^{N+sp}}\,dx\,dy \rightarrow \int_{\mathbb{R}^N \times \mathbb{R}^N} \dfrac{|u(x)-u(y)|^{p}}{|x-y|^{N+sp}}\,dx\,dy.
    \end{aligned}
\end{equation*}
We also have $u_{\mu_{j}}\rightarrow u$ in $C^{\alpha}(\overline{\Omega})$ for $\mu_{j}\rightarrow 0$. So we have 
\begin{equation}\label{nscp33}
	\frac{1}{p} \int_{\mathbb{R}^N \times \mathbb{R}^N} \dfrac{|u(x)-u(y)|^{p}}{|x-y|^{N+sp}}\,dx\,dy +\int_{\Omega}\left( -\frac{\lambda u^{p}_{+}}{p}-\frac{u^{p^{*}_{s}}_{+}}{p^{*}_{s}}\right)\,dx  \ge c_{0}
\end{equation}
and $u$ weakly satisfy \eqref{nscp6} with $\mu=0$.
  From mountain pass Lemma \ref{nscpl2} we have $c_{0}>0.$ Since $u$ satisfy \eqref{nscp33}, this implies that $u$ can not be identically zero. For each $x\in \Omega, \lambda u^{p-1}_{+}(x)+u_{+}^{p^{*}_{s}-1}(x) \ge 0$ and using strong minimum principle and Hopf's lemma \cite{HDLQ}, 
we have
\begin{equation*}
	u(x)\ge c d^{s}(x)>0\thickspace \text{in}\thickspace \Omega.	
\end{equation*}
Now upto a subsequence $u_{\mu_{j}}\rightarrow u$ in $C_{d}^{0,\alpha}(\overline{\Omega})$ for some $\alpha \in (0,s).$ This implies that $u_{\mu_{j}}>0$ in $\Omega$ for sufficiently large $j$. So we conclude that for small $\mu,\thinspace u_{\mu}>0 $ is solution of the problem \eqref{nscp3} with $DE_{\mu}(u_{\mu}) =0$ and $E_{\mu}(u_{\mu})=c_{\mu}$ where $c_{\mu}$ is as given in Lemma \ref{nscpl2} viz.,
\begin{equation*}
	c_{\mu}:=\inf_{\gamma \in \Gamma}	\max_{0\leq t\leq 1} E_{\mu}(\gamma(t))
\end{equation*}
where $\Gamma=\left\lbrace \gamma\in C([0,1],D^{s,p}_{0}(\Omega)):\gamma(0)=0,\gamma(1)=Ru_{\epsilon,\delta} \right\rbrace$. Now we can show that $u_{\mu}$ is ground state solution by following the ideas of \cite[Theorem 1.1]{PSS}. This completes the proof of the Theorem \ref{nscpt1}.\hfill\qed.

%through the fibering map $g:[0,\infty)\rightarrow \mathbb{R}$ given by $g(t):=E_{\mu}(tw)$ for $w\in \mathcal{N}_{\mu}$ i.e. for $t \geq 0$
%\begin{equation*}
 %   \begin{aligned}
  %    	g(t)=E_{\mu}(tw)= \frac{t^{p}}{p} \int_{\mathbb{R}^{N} \times \mathbb{R}^{N}} \dfrac{|w(x)-w(y)|^{p}}{|x-y|^{N+sp}}\,dx\, dy &- \frac{\lambda t^{p}}{p}\int_{\Omega}w^{p}\,dx \\&- \frac{t^{p^{*}_{s}}}{p^{*}_{s}} \int_{\Omega} w^{p^{*}_{s}}\,dx
%+ \mu t\int_{\Omega} w\,dx. 
 %\end{aligned}
%\end{equation*}
%Therefore similar to the ideas of \cite[Theorem 1.1]{PSS}, we can conclude that
% \begin{equation*}
% 	\max_{t\ge 0}E_{\mu}(tw)=E_{\mu}(w)>0
% \end{equation*}
%  For fixed $t_{0}>\max\{1,R\}$, we define the map $\gamma :[0,1]\rightarrow D^{s,p}_{0}(\Omega)$ as $\gamma(t):=(t_{0}w)t$ that satisfies $E_{\mu}(t_{0}w)\leq 0$ and so $\gamma \in \Gamma$. Hence
% \begin{equation*}
% 	c_{\mu}\leq \max_{0\leq t\leq 1}E_{\mu}(\gamma(t))\leq \max_{t\ge 0}E_{\mu}(tw)=E_{\mu}(w).
% \end{equation*}
% Since $w\in \mathcal{N}_{\mu}$ is arbitrary, we get $c_{\mu}\leq \inf_{w\in \mathcal{N}_{\mu}}E_{\mu}(w)$. This completes the proof of the Theorem \ref{nscpt1}.

\begin{appendices}
\section{}
We first state some of the well known auxiliary results.\\
\underline{\textbf{Auxiliary results}}
\begin{lemma}
\label{lemC1}
Let $1<p<\infty$ and $\ba\geq1.$	For every $a,b,t\geq0,$ it holds that $$|a-b|^{p-2}(a-b)(a_t^{\ba}-b_t^{\ba})
\geq \frac{{\ba}p^p}{({\ba}+p-1)^p}\left|a_t^{\frac{{\ba}+p-1}{p}}-b_t^{\frac{{\ba}+p-1}{p}}\right|^p,$$
where $a_t=\min\{a,t\}$ and $b_t=\min\{b,t\}.$
\end{lemma}
\begin{lemma}
\label{lem2}
Let $1<p<\infty$ and $\phi:\RR\to\RR$ be a  differentiable convex function. Then
\begin{align*}
&|a - b|^{p-2} (a - b)\left[
c~|\phi'(a)|^{p-2} \phi'(a) - t ~|\phi'(b)|^{p-2} \phi'(b)\right]\nonumber\\
&\hspace{6cm}\geq |\phi(a) - \phi(b)|^{p-2} (\phi(a)-\phi(b)) (c - t),
\end{align*}
for every $a, b \in \RR$ and every $c,t \geq 0.$
\end{lemma}
Now, we state the regularity result and prove it.
\begin{theorem}
\label{anscpt3}
Let $\Omega\subset\RR^N, N\geq2$ be a bounded domain and 
$u\in 
D^{s,p}_{0}(\Omega)$ weakly solve $(-\Delta)_{p}^s u =f_\mu( x,u)$ in $\Omega$ for a parameter $\mu>0$ and 
$f$ satisfying the following growth condition
\begin{equation}\label{anscp29}
			|f_\mu(x,t)|\leq C_{1}\left( 1+|t|^{p^{*}_{s}-1}\right) ,
\end{equation}
where $C_{1}>0$ is constant. If $u$ is uniformly bounded in $D^{s,p}_{0}(\Omega)$ and $ \int_{E} |u|^{p^{*}_{s}} \,dx \rightarrow 0 \thickspace \text{as} \thickspace |E|\rightarrow 0,\thickspace \text{uniformly in} \thickspace \mu$ then 
there exist two positive constants $C_*$ and $ C^*$ depending upon $s,p,N,\Om$ such that

\begin{equation*}
    \begin{aligned}
          \|u\|_{L^{\infty}(\Om)}\leq \left(C_*\right)^{\DD\frac{1}{p_{s}^*-p}}\left(C^*\right)^{\DD\frac{p-1}{\sqrt {p}(\sqrt{p_{s}^*}-\sqrt {p})}}\|u\|_{L^{  p_s^*}(\Om)}.
    \end{aligned}
\end{equation*}
that is, $\|u\|_{L^\infty(\Omega)}$are uniformly bounded.

% Let $\Omega$ be a bounded and $u\in 
% D^{s,p}_{0}(\Omega)$ weakly solve $(-\Delta)_{p}^s u=f(x,u)$ in $\Omega$ for $f$ satisfying the following growth condition
% \begin{equation}\label{anscp29}
% 			|f(x,t)|\leq C_{1}\left( 1+|t|^{p^{*}_{s}-1}\right) ,
% \end{equation}
% where $C_{1}>0$ is constant.
% Then, any weak solution $u\in D_0^{s,p}(\Om)$ belongs to  $ L^\infty(\Om).$
% Moreover, there exist two positive constants $C_*$ and $ C^*$ depending upon $s,p,\mu,N,\Om$ such that

% \begin{equation*}
%     \begin{aligned}
%           \|u\|_{L^{\infty}(\Om)}\leq \left(C_*\right)^{\DD\frac{1}{p_{s}^*-p}}\left(C^*\right)^{\DD\frac{p-1}{\sqrt {p}(\sqrt{p_{s}^*}-\sqrt {p})}}\|u\|_{L^{  p_s^*}(\Om)}.
%     \end{aligned}
% \end{equation*}
\end{theorem}
\begin{proof}
For every $0 < \e <<1, $
we define the smooth convex Lipschitz function $$h_\e(t)=(\e^2+t^2)^{\frac{1}{2}}$$ and take $\psi|h_\e'(u)|^{p-2}h_\e'(u)$ as the test function
in weak formulation, where $\psi\in C_c^\infty(\Om), \psi>0.$ 
By Lemma \ref{lem2}, we obtain
		\begin{align}\label{1}
			&\int_{{\RR^N}}\int_{\RR^N}\frac{|
				h_\e(u(x))-h_\e(u(y))|^{p-2}(h_\e(u(x))-h_\e(u(y)))(\psi(x)-\psi(y))}{|
				x-y|^{N+sp}}dxdy\n&\hspace{5.7cm}\leq\int_{\Om}|f(x,u)|~ |h_\e'(u(x))|^{p-1}\psi(x)\;dx.
		\end{align}
		Since $h_\e(t)$ converges to $h(t)=|t|$ as $\e\to 0^+$ and $|h_\e'(t)|\leq 1$, 
We note that, using Young's inequality and Lipschitz continuity of $h_\epsilon,$
  \begin{align*}
			&\frac{|
				h_\e(u(x))-h_\e(u(y))|^{p-2}(h_\e(u(x))-h_\e(u(y)))(\psi(x)-\psi(y))}{|
				x-y|^{N+sp}}\\
    &\hspace{4cm}\geq -\frac{|h_\e(u(x))-h_\e(u(y))|^{p-1}|\psi(x)-\psi(y)|}{|x-y|^{N+sp}}\\
    &\hspace{4cm}\geq -\frac{(p-1)}{p}\frac{|h_\e(u(x))-h_\e(u(y))|^{p}}{|x-y|^{N+sp}}
    -\frac{1}{p}\frac{|\psi(x)-\psi(y)|}{|x-y|^{N+sp}}~~~\\
    &\hspace{5.4cm}\geq -\frac{(p-1)}{p}\frac{|u(x)-u(y)|^{p}}{|x-y|^{N+sp}}
    -\frac{1}{p}\frac{|\psi(x)-\psi(y)|}{|x-y|^{N+sp}}~~~
		\end{align*}
 By using generalized Fatou's lemma in \eqref{1}, we get
		\begin{align}\label{2}
			&\int_{{\RR^N}}\int_{\RR^N}\frac{\Big|
				|u(x)|-|u(y)|\Big|^{p-2}\Big(|u(x)|-|u(y)|\Big)\Big(\psi(x)-\psi(y)\Big)}{|
				x-y|^{N+sp}}dxdy\n&\hspace{7.4cm}\leq\int_{\Om}\;|f(x,u)|\psi(x)\;dx.
		\end{align}  Since $C_c^\infty(\Om)$ is dense in $D_0^{s,p}(\Om),$ \eqref{2} holds true, for $0\leq\psi\in D_0^{s,p}(\Om).$ Next, for $l\in\mathbb N$, we define $u_l=\min\{l,|u(x)|\}.$ Clearly $u_l\in D_0^{s,p}(\Om).$ For $k\geq1,$ let us set
		$\beta:=kp-p+1.$ So, $\beta>1.$  Choosing $\psi=u_l^{\ba }$ in \eqref{2} and using Lemma \ref{lemC1}, we obtain
		\begin{align}\label{3}
			\frac{\beta p^p}{(\ba+p-1)^p}\int_{\RR^N}\int_{\RR^N}&\frac{\left|(u_l(x))^{\frac{\ba+p-1}{p}}-(u_l(y))^{\frac{\ba+p-1}{p}}\right|^p}{|x-y|^{N+sp}}dxdy\n
			&\hspace{3.4cm}\leq\int_{\Om}\;|f(x,u)|(u_l(x))^{\ba}\;dx.
		\end{align}
		By observing  that   $$\frac{1}{\beta}\left({\frac{\ba+p-1}{p}}\right)^p\leq \left({\frac{\ba+p-1}{p}}\right)^{p-1}, 
		\text{~~for large~} \ba$$ and using the  the relation $k=\frac{\ba+p-1}{p}$ along with continuous embedding $D_0^{s,p}(\Om)\hookrightarrow L^{p_s^*}(\Om)$, from \eqref{3}, we get
		\begin{align}\label{4}
			\left\|u_l^{k}\right\|^p_{L^{p_s^*}(\Om)}\leq \frac {(k)^{p-1}}{ S_{s,p}^p}\int_{\Om}\;|f(x,u)|(u_l(x))^{\ba}\;dx,
		\end{align}
		where the best Sobolev constant $S_{s,p}$.
		Now we will estimate the right-hand side  in \eqref{4}. Using the fact $u_l\leq |u|,$ we deduce
		\begin{align}\label{5.0.0}
		    &	\int_{\Om}\;|f(x,u)|(u_l(x))^{\ba}~dx
				\leq C_1\left(\int_\Om|u_l|^\beta\, dx+\int_\Om|u|^{p_{s}^*-1} |u_l^{\ba}|\,dx\right)\nonumber
				\\&
			=C_1\Bigg[\int_{\Om\cap\{|u|<\Lambda\}}|u_l|^{\beta} dx+\int_{\Om\cap\{|u|\geq\Lambda\}}|u_l|^{\beta} dx
			+\int_{\Om\cap\{|u|<\Lambda\}}|u|^{p_{s}^*-1} |u_l^{\ba}|\,dx\nonumber
			\\&\hspace{9cm}
			+\int_{\Om\cap\{|u|\geq\Lambda\}}|u|^{p_{s}^*-1} |u_l^{\ba}|\,dx \Bigg]\nonumber
			\\
			&\leq C_1\Bigg[\int_{\Om\cap\{|u|<\Lambda\}}|u|^{\beta}\,dx+\int_{\Om\cap\{|u|\geq\Lambda\}}|u|^{p_{s}^*+\beta-1}\,dx
			+\int_{\Om\cap\{|u|<\Lambda\}}|u|^{p_{s}^*+\beta-1} \,dx\nonumber
			\\&\hspace{9.2cm}
			+\int_{\Om\cap\{|u|\geq\Lambda\}}|u|^{p_{s}^*+\beta-1} \,dx\Bigg]
			\n&\leq C_1\Bigg[\Lambda^{1-p}\int_{\Om\cap\{|u|<\Lambda\}}|u|^{p+\beta-1} \,dx+\Lambda^{p_{s}^*-p}\int_{\Om\cap\{|u|<\Lambda\}} |u|^{p+\beta-1}\,dx\nonumber
			\n&\hspace{7.8cm}+2\int_{\Om\cap\{|u|\geq\Lambda\}}|u|^{p_{s}^*-p}\; |u|^{p+\beta-1}\,dx\Bigg]
			\n&\hspace{0.5cm}
			\leq C_1\Big[(\Lambda^{1-p} +\Lambda^{p_{s}^*-p})\;\|u\|_{L^{kp}(\Om)}^{kp}+2\int_{\Om\cap\{|u|\geq\Lambda\}}|u|^{p_{s}^*-p}\; |u|^{p+\beta-1}\,dx\Big],
		\end{align}
% 			\begin{align*}
% 				&\int_{\Om}\;|f(x,u)|(u_l(x))^{\ba}~dx\n
% 				&\leq C_1\left(\int_\Om|u_l|^\beta\, dx+\int_\Om|u|^{p_{s}^*-1} |u_l^{\ba}|\,dx\right)\n
% 				&=C_1\Bigg[\int_{\Om\cap\{|u|<\Lambda\}}|u_l|^{\beta} dx+\int_{\Om\cap\{|u|\geq\Lambda\}}|u_l|^{\beta} dx\n&\qquad+\int_{\Om\cap\{|u|<\Lambda\}}|u|^{p_{s}^*-1} |u_l^{\ba}|\,dx+\int_{\Om\cap\{|u|\geq\Lambda\}}|u|^{p_{s}^*-1} |u_l^{\ba}|\,dx
% 				\Bigg]
% 		\end{align*} 
% 		\begin{align}\label{5.0.0}
% 				&\leq C_1\Bigg[\int_{\Om\cap\{|u|<\Lambda\}}|u|^{\beta}\,dx+\int_{\Om\cap\{|u|\geq\Lambda\}}|u|^{p_{s}^*+\beta-1}\,dx\n
% 				&\;+\int_{\Om\cap\{|u|<\Lambda\}}|u|^{p_{s}^*+\beta-1} \,dx+\int_{\Om\cap\{|u|\geq\Lambda\}}|u|^{p_{s}^*+\beta-1} \,dx\Bigg]\n
% 				&\leq C_1\Bigg[\Lambda^{1-p}\int_{\Om\cap\{|u|<\Lambda\}}|u|^{p+\beta-1} \,dx+\Lambda^{p_{s}^*-p}\int_{\Om\cap\{|u|<\Lambda\}} |u|^{p+\beta-1}\,dx\n&\qquad\qquad\qquad\qquad+2\int_{\Om\cap\{|u|\geq\Lambda\}}|u|^{p_{s}^*-p}\; |u|^{p+\beta-1}\,dx\Bigg]\n&\leq C_1\Big[(\Lambda^{1-p} +\Lambda^{p_{s}^*-p})\;\|u\|_{L^{kp}(\Om)}^{kp}+2\int_{\Om\cap\{|u|\geq\Lambda\}}|u|^{p_{s}^*-p}\; |u|^{p+\beta-1}\,dx\Big],
% 		\end{align}
where $\Lambda>1$ will be chosen later, $C_1>0$ is a constant. By plugging  \eqref{5.0.0} into \eqref{4} and applying Fatou's lemma, we get
		\begin{align}\label{6}
			\|u\|_{L^{k p_s^*}(\Om)}^{k p}
			&\leq C_{1} \frac {k^{p-1}}{ (S_{s,p})^p}\Bigg[(\Lambda^{1-p} +\Lambda^{p_{s}^*-p})\;\|u\|_{L^{kp}(\Om)}^{kp}\nonumber
			\\&\hspace{4.2cm}
			+2\int_{\Om\cap\{|u|\geq\Lambda\}}|u|^{p_{s}^*-p}\; |u|^{p+\beta-1}\,dx\Bigg].
		\end{align}
		Now we estimate the second term on the right-hand side in \eqref{6}. For this, using H\"older inequality for constant exponents, we obtain
		\begin{align}	
			\int_{\Om\cap\{|u|\geq\Lambda\}}|u|^{p_{s}^*-p}\; |u|^{kp}\,dx\leq& C_2\left(\int_{\Om\cap\{|u|\geq\Lambda\}}|u|^{p_s^*} dx\right)^{\frac{p_{s}^*-p}{p_s^*}}\left(\int_{\Om}|u|^{k p_s^*} dx\right)^{\frac{p}{p_s^*}}
			\n&\hspace{3cm}
			=C(\Lambda,u)
			\|u\|_{L^{k p_s^*}(\Om)}^{kp},	
		\end{align}
		where $C_2>0$ is a constant and $C(\Lambda,u)=C_2\left(\int_{\Om\cap\{|u|\geq\Lambda\}}|u|^{p_s^*} dx\right)^{\frac{p_{s}^*-p}{p_s^*}}.$ Using uniform boundedness and uniform integrability of $u$ , $C(\La,u)$ only depends on $u$ for $\La$ large so we write this as $C(\La)$ and get
		\begin{equation}
		\label{7}
		    \begin{aligned}
		        \int_{\Om\cap\{|u|\geq\Lambda\}}|u|^{p_{s}^*-p}\; |u|^{kp}\,dx \leq C(\Lambda)
			\|u\|_{L^{k p_s^*}(\Om)}^{kp}  
		    \end{aligned}
		\end{equation}
		Combining \eqref{6} and \eqref{7}, we have
		\begin{align}\label{8}
			\|u\|_{L^{k p_s^*}(\Om)}^{k p}\leq C_{1} \frac {k^{p-1}}{ (S_{s,p})^p}\left[(\Lambda^{1-p}+\Lambda^{p_{s}^*-p} )\|u\|_{L^{kp}(\Om)}^{kp}
			+ 2C(\Lambda)
			\|u\|_{L^{k p_s^*}(\Om)}^{k p}\right].
		\end{align} 
		Now by applying Lebesgue dominated convergence theorem in \eqref{7}, we  can choose $\Lambda>1$ large enough such that   $\DD C(\Lambda)< \frac{ (S_{s,p})^p} {4\tilde{C}(k)^{p-1}}$ and
		hence, from \eqref{8}, it follows that
		\begin{align}\label{9}
			\|u\|_{L^{k p_s^*}(\Om)}\leq \left(C_*^{\frac{1}{k}}\right)^{\frac{1}{p}} (k^{\frac{1}{k}})^{\frac{p-1}{p}}\|u\|_{L^{kp}(\Om)},
		\end{align}
		where $\DD C_*= \frac {2 C_{1}~(\Lambda^{1-p}+\Lambda^{p_{\mu,s}^*-p})}{ (S_s)^p}>1.$ Next, we start bootstrap argument on \eqref{9}.\\
			Choose  $k=k_1:=\frac{p_{s}^*}{p}>1$ as the first iteration. Thus,  \eqref{9} yields that
		\begin{align}\label{it1}
			\|u\|_{L^{k{_1} p_s^*}(\Om)}\leq \left(C_*^{\frac{1}{k{_1}}}\right)^{\frac{1}{p}} (k_{1}^{\frac{1}{k_{1}}})^{\frac{p-1}{p}}\|u\|_{L^{ p_s^*}(\Om)}.
		\end{align}
		Again by taking $k=k_2:=k_1 \frac{p_{s}^*}{p}$ as the second iteration in \eqref{9} and then inserting   \eqref{it1} in \eqref{9}, we get
		\begin{align}\label{it2}
			\|u\|_{L^{k_2 p_s^*}(\Om)}&\leq \left(C_*^{\frac{1}{k_2}}\right)^{\frac{1}{p}} \left[(k_2)^{\frac{1}{k_2}}\right]^{\frac{p-1}{p}}\|u\|_{L^{ k_1 p_s^*}(\Om)}\n
			&\leq \left(C_*^{\frac{1}{k_1}+\frac{1}{k_2}}\right)^{\frac{1}{p}} \left[(k_1)^{\frac{1}{k_1}}.(k_2)^{\frac{1}{k_2}}\right]^{\frac{p-1}{p}}\|u\|_{L^{  p_s^*}(\Om)}.
		\end{align}
		In this fashion,  taking $k=k_n:=k_{n-1}\frac{p_{s}^*}{p}$ as the $n^{th}$ iteration and iterating for $n$ times, we obtain
		\begin{align}\label{itn}
			\|u\|_{L^{k_n p_s^*}(\Om)}&\leq \left(C_*^{\frac{1}{k_n}}\right)^{\frac{1}{p}} \left[(k_n)^{\frac{1}{k_n}}\right]^{\frac{p-1}{p}}\|u\|_{L^{ k_{n-1} p_s^*}(\Om)}\n
			&\leq \left(C_*^{\DD\frac{1}{k_1}+\frac{1}{k_2}\cdots +\frac{1}{k_n}}\right)^{\frac{1}{p}} \left[(k_1)^{\DD\frac{1}{k_1}}.(k_2)^{\DD\frac{1}{k_2}}\cdots (k_n)^{\DD\frac{1}{k_n}}\right]^{\frac{p-1}{p}}\|u\|_{L^{  p_s^*}(\Om)}\n
			&=\left(C_*^{\DD\sum_{j=1}^{n}{\frac{1}{k_j}}}\right)^{\frac{1}{p}}\left(\prod_{j=1}^{n}\left(k_j^{\sqrt{1/{k_j}}}\right)^{\sqrt{1/{k_j}}}\right)^{\frac{p-1}{p}}\|u\|_{L^{  p_s^*}(\Om)},
		\end{align}
		where $k_j=\left(\frac{p_{s}^*}{p}\right)^j.$ Since $\frac{p_{s}^*}{p}>1,$ we have $k_j^{\DD\sqrt{1/{k_j}}}>1,$ for all $j\in\mathbb N$ and
		$$\lim_{j\to+\infty}k_j^{\DD\sqrt{1/{k_j}}}=1.$$
		Hence, it follows that there exists a constant $C^*>1,$ independent of $j,n\in\mathbb N$ such that $k_j^{\DD\sqrt{1/{k_j}}}<C^*$ and thus, \eqref{itn} gives
		\begin{align}\label{10}
			\|u\|_{L^{k_n p_s^*}(\Om)}&\leq \left(C_*^{\DD\sum_{j=1}^{n}{\frac{1}{k_j}}}\right)^{\frac{1}{p}}\left({C^*}^{\DD\sum_{j=1}^{n}\sqrt{1/{k_j}}}\right)^{\frac{p-1}{p}}\|u\|_{L^{  p_s^*}(\Om)}.
		\end{align}
		%	As limit $n\to+\infty,$ the sum of the following geometric series are given as:
		Observe that
		$$\sum_{j=1}^{\infty}{\frac{1}{k_j}}=\sum_{j=1}^{n}\left({\frac{p}{p_{s}^*}}\right)^j=\frac{p/p_{s}^*}{1-p/p_{s}^*}=\frac{p}{p_{s}^*-p}$$ and
		$$\sum_{j=1}^{\infty}{\frac{1}{\sqrt{k_j}}}=\sum_{j=1}^{n}{\left(\sqrt{\frac{p}{p_{s}^*}}\right)^{j}}=\frac{\sqrt p}{\sqrt{ p_{s}^*}-\sqrt p},$$
		from  \eqref{10}, we get that
		\begin{align}\label{contr}
			\|u\|_{L^{\nu_n}(\Om)}\leq \left(C_*\right)^{\DD\frac{1}{p_{s}^*-p}}\left(C^*\right)^{\DD\frac{p-1}{\sqrt {p}(\sqrt{p_{s}^*}-\sqrt {p})}}\|u\|_{L^{  p_s^*}(\Om)},
		\end{align} where $\nu_n:=k_n p_s^*.$ Note that, $\nu_n\to+\infty$ as $n\to+\infty.$ Therefore, we claim that  
		\begin{align}\label{claim}
			\|u\|_{L^{\infty}(\Om)}=\|u\|_{L^{\infty}(\Om)}\leq \left(C_*\right)^{\DD\frac{1}{p_{s}^*-p}}\left(C^*\right)^{\DD\frac{p-1}{\sqrt {p}(\sqrt{p_{s}^*}-\sqrt {p})}}\|u\|_{L^{  p_s^*}(\Om)}.
		\end{align}
		Indeed, if not, let us assume $\|u\|_{L^{\infty}(\Om)}>C_3\|u\|_{L^{  p_s^*}(\Om)}, $ where $$C_3=\left(C_*\right)^{\DD\frac{1}{p_{s}^*-p}} \left(C^*\right)^{\DD\frac{p-1}{\sqrt {p}(\sqrt{p_{s}^*}-\sqrt {p})}}.$$
		Then there exists $C_4>0$ and a subset $\mathcal{S}$ of $\Om$ with $\mathcal{|S|}>0$  such that 
		$$u(x)>C_3\|u\|_{L^{  p_s^*}(\Om)}+C_4,\text{~~for } x\in\mathcal{S}.$$ The above implies
		\begin{align*}
			\DD\liminf_{\nu_n\to+\infty}\left(\int_{\Om}|u(x)|^{\nu_n}dx\right)^{\frac{1}{\nu_n}}
			&\geq \DD\liminf_{\nu_n\to+\infty}\left(\int_{\mathcal{S}}|u(x)|^{\nu_n}dx\right)^{\frac{1}{\nu_n}}\\&\geq\DD\liminf_{\nu_n\to+\infty}\left(C_3\|u\|_{L^{  p_s^*}(\Om)}+C_4\right)\left(\mathcal{|S|}\right)^{\frac{1}{\nu_n}}\\
			&=\DD\liminf_{\nu_n\to+\infty}\left(C_3\|u\|_{L^{  p_s^*}(\Om)}+C_4\right),
		\end{align*}
		a contradiction to \eqref{contr}. Therefore, \eqref{claim} holds and hence, $u\in L^{\infty}(\Om).$ 
\end{proof}
\end{appendices}

 \bibliographystyle{plain}
	\bibliography{ref_utt}

	\end{document}